\def\boxit#1{\vbox{\hrule\hbox{\vrule\kern6pt
          \vbox{\kern6pt#1\kern6pt}\kern6pt\vrule}\hrule}}

\newcommand{\calI}{{\cal I}}
\newcommand{\calJ}{{\cal J}}

\newcommand{\PK}{\textbf{\rm PK}}

\newcommand{\bff}{{\bf f}}

\newcommand{\bfMult}{{\bf Mult}}

\newcommand{\veps}{\varepsilon}

\newcommand{\todr}{\stackrel{\mathrm{D}}{\longrightarrow}}

\newcommand{\eqdr}{\stackrel{\mathrm{D}}{=}}

\newcommand{\R}{\Bbb{R}}

\newcommand{\N}{\Bbb{N}}

\newcommand{\rmd}{{\rm d}}
\newcommand{\rmi}{{\rm i}}
\newcommand{\halmos}{\quad\hfill\mbox{$\Box$}}

\newcommand{\wtilde}{\widetilde}

\newcommand{\wh}{\widehat}

\newcommand{\dto}{\downarrow}

\newcommand{\be}{\begin{equation}}
\newcommand{\ee}{\end{equation}}
\newcommand{\bea}{\begin{eqnarray}}
\newcommand{\eea}{\end{eqnarray}}
\newcommand{\bean}{\begin{eqnarray*}}
\newcommand{\eean}{\end{eqnarray*}}
\newcommand{\ben}{\begin{equation*}}
\newcommand{\een}{\end{equation*}}
\newcommand{\ba}{\begin{align}}
\newcommand{\ea}{\end{align}}

\def\nexto{\kern -0.54em}

\newcommand{\PP}{\textbf{\rm P}}
\newcommand{\EE}{\textbf{\rm E}}
\newcommand{\PD}{\textbf{\rm PD}}

\newcommand{\bfM}{{\bf M}}
\newcommand{\bfm}{{\bf m}}
\newcommand{\bfp}{{\bf p}}
\newcommand{\bfq}{{\bf q}}
\newcommand{\bfQ}{{\bf Q}}

\newcommand{\bfu}{{\bf u}}

\newcommand{\lf}{\lfloor}
\newcommand{\rf}{\rfloor}
\newcommand{\lc}{\lceil}
\newcommand{\rc}{\rceil}

\documentclass[11pt]{article}
\usepackage{graphicx}
\usepackage[shortlabels]{enumitem}
\usepackage[utf8]{inputenc}
\usepackage[english]{babel}
\usepackage{amssymb,amsmath,amsthm, amsfonts}
\usepackage[T1]{fontenc}
\usepackage{scrextend}
\usepackage[text={7in,10in},centering]{geometry}
\usepackage{cases}
\usepackage{grffile}
\usepackage{caption}

\RequirePackage[OT1]{fontenc}
\RequirePackage{amsthm,amsmath}
\RequirePackage[numbers]{natbib}
\RequirePackage[colorlinks,citecolor=blue,urlcolor=blue]{hyperref}

\usepackage[T1]{fontenc} 
\usepackage{lmodern} 
\usepackage[ddmmyyyy]{datetime}

\numberwithin{equation}{section}
\theoremstyle{plain}

\newtheorem{theorem}{Theorem}[section]

\newtheorem{corollary}{Corollary}[section]
\newtheorem{lemma}{Lemma}[section]

\begin{document}
\bibliographystyle{plain}

\title{Limit Theorems  for the Pitman-Yor Frequency Spectrum}

\author{ Ross Maller$^a$ and Soudabeh Shemehsavar$^{b}$
\thanks{Email:  Ross.Maller@anu.edu.au;  soudabeh.shemehsavar@murdoch.edu.au 
($^*$corresponding author).}}

\maketitle

\begin{abstract}\ 
We  derive a  general distribution formula  , 
 applicable to  a wide variety of  Gibbs-type partitions 
and use it to obtain   large sample results for linear combinations of  the 
 component frequency spectrum $(M_{jn})_{1\le j\le n}$ 
associated with a random partitioning of $\{1,2,\ldots, n\}$.
The two-parameter Pitman-Yor sampling model  is analysed in detail
and asymptotic distributions of sums of the form 
$\sum _{j=\lf \lambda n\rf}^{\lf \mu n\rf} M_{jn}$, $0<\lambda\le \mu\le 1$, are obtained.
Our results suggest a possible functional limit theorem for  
$\sum _{j=\lf \lambda n\rf}^{n} M_{jn}$.. 
\end{abstract}

\noindent {{\bf Keywords:}
Random partitions; component frequency spectrum;
Pitman-Yor sampling formula; 
shape of Young diagram; 
Gibbs distribution.

{}\null
\vskip-1cm 
\section{Introduction}\label{intro} 
Models for the distribution of partition sizes in a random partitioning of a finite set are important in a variety of areas, including statistical mechanics, combinatorics, genetics, and others, 
and have received close and continuing attention over the years. 
As in   \cite{Pitman2006}, p.14,  we partition the set $\N_n=\{1,2,\ldots,n\}$ into blocks $(M_{jn})_{1\le j\le n}$,
where   $M_{jn}$ is the number of partitions of size $j$.
These are collected into the vector\footnote{Vectors and matrices are denoted in boldface, with a superscript $T$ for transpose.
Pitman   \cite{Pitman2006}, p.15, uses the notation $|\Pi_{nj}|$ where we have $M_{nj}$, with $|\Pi_n|$ for our $K_n$.} 
$\bfM_n= (M_{1n}, M_{2n},  \ldots, M_{nn})$.
We denote by  $K_n =\sum_{j=1}^n M_{jn}$   the total number of blocks.
The  random variable (``rv")
$K_n$ takes values $k\in \N_n$, $n\in\N$, 
 and  $\bfM_n$ takes values in the set
\be\label{4.4a0}
A_{kn}=\Big\{\bfm= (m_1,\ldots,m_n)\in (\{ 0\}\cup \N_n)^n, \, \sum_{j=1}^n m_j=k,\, \sum_{j=1}^n jm_j=n\Big\}.
\ee
In the theory of decomposable random combinatorial structures, $\bfM_n$ is called the {\it component frequency spectrum}; in genetics, it is the 
{\it allele frequency spectrum}.

Suppose   $(\bfM_n, K_n)$
has a distribution of the form:
\be\label{N1}
\PP(\bfM_n=\bfm,\, K_n=k)
= C_{nk}\prod_{j=1}^n \frac{q_j^{m_j}}{m_j!},
\ n\in \N,\, k\in \N_n   
\ee
  (\cite{Pitman2006}, p.26),
 where $q_j>0$, $1\le j\le n$, 
 $\bfm= (m_1,\ldots,m_n)\in A_{kn}$,  and $C_{nk}>0$
  is a normalising constant depending on $k$ and $n$ and possibly other parameters.  $K_n =\sum_{j=1}^n M_{jn}$  is a deterministic function of $\bfM_n$ but it is useful  to include it explicitly in the formulation \eqref{N1}.
  Asymptotic properties of $K_n$ are known for a variety of  models, but our aim is to study the limiting behaviour  of the entire spectrum, joint with $K_n$, assuming just \eqref{N1}.
   To do this we adopt a Cram\'er-Wold device,  considering linear combinations of the elements of $\bfM_n$; specifically, we analyse the convergence of functionals like $\sum_{j=1}^n 
u_{jn} M_{jn}$, for $n$-vectors $\bfu_n =(u_{jn}, 1\le j\le n)\in\R^n$.
A distinctive feature is that the length of the vector grows with $n$.
  
  A  special case of interest is the 
Pitman-Yor  sampling model in \cite{PY1997}.   
The asymptotic (large sample) behaviour of $K_n$ is known 
from \cite{Pitman2006}, p.68.
To add in  the frequency spectrum  we analyse $K_n$ jointly with  functionals like
$\sum_{j=1}^n f(j/n)M_{jn}$, for suitable
functions $f$ on $[0,1]$ (bounded variation suffices, see
the proof of Lemma \ref{ylem}), 
 with  $\lim_{y\dto 0}f(y)=f(0)=0$, satisfying  a certain growth condition at 0.    

We begin with a general result on 
Gibbs-type partitions
which is specialized in  Section \ref{news} for 
the Pitman-Yor model. 
 Proofs are deferred to Section \ref{proofs} and an  Appendix. 

\section{Gibbs-Type Partitions}\label{gen}
 Recalling \eqref{N1} and the setup in Section \ref{intro},  
    let  $q_{+n}= \sum _{j=1}^nq_j$ and define the $n$-vector
 \be\label{qnvec}
 \bfq_n=(q_{1n}, \ldots, q_{nn}) 
 =\frac{1}{q_{+n} }  (q_{1}, \ldots, q_{n}),
 \ee
 which will serve as a centering vector.
 Let 
 $\bfu_n =(u_{jn})_{1\le j\le n}=
(u_{1n}, \ldots, u_{nn})$ be an arbitrary $n$-vector.
Our results stem from the following 
 formula for the 
 moment generating function  (mgf) of   $\bfM_n/K_n$, joint with 
$K_n$,  which is proved in  Section \ref{proofs}.

\begin{theorem}\label{thm2N}
 Assume \eqref{N1}. We have for  $\nu\in\R$, 
$\bfu_n\in\R^n$ and $k\in \N_n$, 
  \bea\label{ecenN}
  &&
\EE\Big(
\exp\Big(\nu\bfu_n^T\Big(\frac{\bfM_n}{K_n}-\bfq_n\Big)\Big); K_n=k\Big)
\cr
&&  \hskip1cm 
=
\EE\Big(\exp\Big(
  \frac{ \nu }{k} \sum_{i=1}^{k}
   \big(V_{in}^{(\bfu_n)}- \EE(V_{in}^{(\bfu_n)})\big)\Big)\Big)\times
\frac{\PP\Big(\sum_{i=1}^{k} \wh X_{in}^{(\nu\bfu_n)}
=n\Big)}
{ \PP\Big(\sum_{i=1}^{k} \wh X_{in}^{(0)}=n\Big)}
 \times \PP(K_n=k).
 \eea
 The 
 $\big(\wh X_{in}^{(\nu{\bfu_n}  )}\big)_{1\le i\le k}$ are i.i.d. with the distribution 
\be\label{HM8Nn}
\PP\big(\wh X_{1n}^{(\nu{\bfu_n}  )}=j\big) 
=
\frac{q_{j}e^{ \nu u_{jn}/k} }
{\sum_{\ell=1}^n q_{\ell } e^{ \nu u_{\ell n}/k}}, \ 1\le j\le n,
\ee
the  $\big(\wh X_{in}^{(0 )}\big)_{1\le i\le k}$ are i.i.d. with
\be\label{H0}
\PP\big(\wh X_{1n}^{(0 )}=j\big) 
= q_{jn} 
:= \frac{q_j}{q_{+n}}
=
\frac{q_{j} }
{\sum_{\ell=1}^n q_{\ell } },\ 1\le j\le n,
\ee
and the  $(V_{in}^{(\bfu_n)})_{1\le i\le k}$  are i.i.d. with 
 \be\label{dvN}
 P(V_{1n}^{(\bfu_n)}= u_{jn})= q_{jn},\ 1\le j\le n.
 \ee
     The $V_{in}^{(\bfu_n)}$ have 
     \be\label{EV}
      \EE(V_{1n}^{(\bfu_n)})= \bfu_n^T\bfq_n
      \quad {\rm  and}\quad
       {\rm Var}(V_{1n}^{(\bfu_n)})= \bfu_n^T\bfQ_n\bfu_n,
       \ee
 where the matrix $\bfQ_n$ is  the $n\times n$ matrix  with 
diagonal elements $q_{jn}(1-q_{jn})$ and 
off-diagonal elements $-q_{jn} q_{\ell n}$, $1\le j\ne \ell \le n$.
\end{theorem}
 The norming of $\bfM_n$ by the rv $K_n$ in \eqref{ecenN} 
is convenient, making the sum $\sum_{j=1}^n M_{jn}/K_n$ equal to 1, but 
it can be replaced by a deterministic norming
in cases when  $K_n$ converges.
The mgf in \eqref{ecenN} can be written as 
$\EE\big(\exp\big(
 \nu  (V_{1n}^{(\bfu_n)}- \EE(V_{1n}^{(\bfu_n)})\big)\big)$, 
but in its present form \eqref{ecenN} shows that the asymptotics of $\bfM_n$ can be related to the convergence of triangular arrays of rvs, i.i.d. in each row.
The event $\{\wh X_{in}^{(0 )}=1, 1\le i\le k\}$ has probability 
$\prod_{j=1}^kq_{jn}>0$, so the denominator in \eqref{ecenN},  and hence $ \PP(K_n=k)$ (see   \eqref{M4.2N1}), 
 are positive for $k\in \N_n$, $n\in\N$.
  Next we apply these formulae to the Pitman-Yor 2-parameter model.
     
\section{The  Two-Parameter Pitman-Yor  Sampling Formula}\label{news}
The Pitman-Yor \cite{PY1997}, p.897,  sampling formula, depending on parameters 
$0<\alpha<1$,  $\theta>-\alpha$, is 
\be\label{FSFN}
\PP(\bfM_n(\alpha,\theta) =\bfm,\, K_n(\alpha,\theta) =k)
= \frac{n!}{\alpha}
\frac{\Gamma(\theta/\alpha+k)}{\Gamma(\theta/\alpha+1)}
\frac{\Gamma(\theta+1)}{\Gamma(n+\theta)}
\prod_{j=1}^n 
\frac{1}{m_j!} 
\Big( \frac{\alpha \Gamma(j-\alpha)}{j!\Gamma(1-\alpha)} \Big)^{m_j},
\ee
where $1\le k\le n$ and $\bfm=(m_1,\ldots,m_n)\in A_{kn}$
(in \eqref{4.4a0})..
For this model define 
\be\label{qdef}
q_j = \frac{\alpha \Gamma(j-\alpha)}{j!\Gamma(1-\alpha)}, 
\quad
q_{+n}= \sum_{j=1}^n q_{j}
\quad {\rm and}\quad
q_{jn}
= \frac{q_j}{q_{+n} }= \frac{q_j}{ \sum_{\ell=1}^n q_{\ell}},
\ 1\le j\le n.  
\ee
Then  \eqref{FSFN} is of the form \eqref{N1} with  $C_{nk}=C_{nk}(\alpha,\theta)
=n!\Gamma(\theta/\alpha+k)\Gamma(\theta+1)/\alpha\Gamma(\theta/\alpha+1)\Gamma(n+\theta)$
and  the $q_j$ in \eqref{qdef}, so we can apply Theorem \ref{thm2N}. 
Note that $\sum_{j\ge 1}q_j=1$.
It's useful to identify  $\bfM_n$ and $K_n$ by attaching parameters such as the $(\alpha, \theta)$ in \eqref{FSFN}.
  We assume   the function $f$ introduced in Section \ref{intro}   decreases to 0 at 0 fast enough for 
\be\label{ig}
\int_0^1 y^{-\alpha-1}|f(y)|\rmd y<\infty,
\ee
and let $\bff_n$ denote the vector $(f(j/n))_{1\le j\le n}$, which we will take for the $\bfu_n$ in \eqref{ecenN}.
For this model it turns out that the centering in \eqref{ecenN} and the norming by $K_n$  can be dispensed with.
  It is known that for the Pitman-Yor model, $K_n(\alpha,\theta) /n^\alpha$ has a limiting Mittag-Leffler distribution almost surely  (\cite{Pitman2006}, p.68), which explains our choice for the conditioning  variable  in the next result.
   The Mittag-Leffler distribution has a finite mgf equal to the Mittag-Leffler function $E_\alpha(\nu)=
 \sum_{\ell=0}^\infty  \nu^\ell/\Gamma(1+\alpha \ell)$.
 \begin{theorem}\label{2para0}
 For the Pitman-Yor 2-parameter model
  when \eqref{ig}  holds, 
 we have for  $ \nu\in\R$ and $x>0$
  \bea\label{ST}
  &&
\lim_{n\to\infty} n^\alpha
 \EE  \big( \exp\big(\nu  
\bff_n^T \bfM_{n}(\alpha,\theta)\big); 
K_n(\alpha,\theta) =\lfloor x n^\alpha\rfloor \big)\cr
&&
=
\exp\Big(\frac{x\alpha}{\Gamma(1-\alpha)}
   \int_0^1 \big(
e ^{\nu  f(y) }  -1\big) y^{-\alpha-1} \rmd y\Big)
\times
\frac{f_{Y_x^{(\nu,f)}}(1)}{f_{Y_x^{(0)}}(1)}
\times
\frac{\alpha x^{\theta/\alpha}\Gamma(\theta)}{\Gamma(\theta/\alpha)}\times 
g_\alpha(x), 
 \eea
  where    $g_\alpha(x)$ is the Mittag-Leffler density (\cite{{Pitman2006}}, p.11).
 Here each  $Y_x^{(\nu,f)}$ is an infinitely divisible  rv  with support on $[0,\infty)$ and a continuous density $f_{Y_x^{(\nu,f)}}$, which is positive on $(0,\infty)$, and 
  $f_{Y_x^{(\nu,f)}}(1)$ is  the density  at the point 1
 of  $Y_x^{(\nu,f)}$,
given as the absolutely convergent integral
 \be\label{f1}
f_{Y_x^{(\nu,f)}}(1)  =
  \frac{1}{2\pi}\int _{-\infty}^\infty e^{-\rmi \tau}
  \exp\Big(\frac{\alpha x}{\Gamma(1-\alpha)}
  \int_{0}^1
 (e^{ \rmi\tau  y} -1 ) e^{\nu  f(y)} y^{-\alpha-1} \rmd y
 \Big) \rmd \tau.
 \ee
$f_{Y_x^{(0)}}(1)$ is obtained by setting $\nu=0$ in \eqref{f1}.
  \end{theorem}

  This  result shows that the   limit in \eqref{ST} exists finitely and is positive.
  When $\nu=0$ it reduces to 
  \be\label{pn0}
\lim_{n\to\infty}
n^\alpha
\PP\big(K_n(\alpha,\theta) =\lfloor x n^\alpha\rfloor \big)
= \frac{\alpha x^{\theta/\alpha} \Gamma(\theta)}{\Gamma(\theta/\alpha)}\times 
g_\alpha(x),\ x>0,
\ee
which  is the local limit theorem in 
\cite{Pitman1999}, p.21, and 
\cite{{Pitman2006}}, p.72, so Theorem \ref{2para0} generalizes that result. 
  By choosing the function $f$ appropriately in Theorem \ref{2para0} we can get concrete applications   in particular cases,
  for example in the following corollaries and theorem.

Limiting conditional distributions are obtained by dividing out 
$n^\alpha\PP\big(K_n(\alpha,\theta) =\lfloor x n^\alpha\rfloor \big)$
in \eqref{ST}.
The next corollary specializes to the case $f(y)=y^p$, $y>0$,  with $p>\alpha$. 

 \begin{corollary}\label{3parc}
  We have, for $ x>0$, $\nu\in\R$,   $p>\alpha$, the limit
  \be\label{STA}
\lim_{n\to\infty}
   \EE  \Big( \exp\Big(\nu  n^{-p}
\sum_{j=1}^n j^p M_{jn}(\alpha,\theta)\Big)\Big|
K_n(\alpha,\theta) =\lfloor x n^\alpha\rfloor \Big)
=
\frac{f_{Y_x^{(\nu,p)}}(1)}{f_{Y_x^{(0)}}(1)} 
\EE(e^{\nu T_x(\alpha,p)}),\
 \ee
 where 
 $f_{Y_x^{(\nu, p)}}(1)$ and $f_{Y_x^{(0)}}(1)$ are positive and given by \eqref{f1} (with $f(y)=y^p$, and  $\nu=0$ in the second case),  and 
      \be\label{mgd5}
 \EE(e^{\nu T_x(\alpha,p)})=
   \exp\Big(\frac{\alpha x }{p\Gamma(1-\alpha)}
   \int_0^1  \big(
e ^{\nu y}  -1\big) y^{-\alpha/p-1}\rmd y\Big).
  \ee
  \end{corollary}
  Notice that the righthand side of \eqref{STA}  does not depend on $\theta$ and 
the  righthand side of \eqref{mgd5} is the mgf of a stable  subordinator    but with L\'evy measure restricted to $(0,1]$.  
    The particular forms of the  righthand sides  of \eqref{ST} and \eqref{STA} 
  derive from the general representation in Theorem \ref{thm2N}.  
  We can evaluate the expressions
  in special cases, as in the next two results.

  First set  $f(y)={\bf 1}_{\lambda\le y\le \mu}$, where
  $0<\lambda\le \mu\le 1$.
Then  we can deduce   for $S_n^{\lambda,\mu}(\alpha,\theta)  :=
\sum _{j=\lf \lambda n\rf}^{\lf \mu n\rf} M_{jn}(\alpha,\theta) $
the following corollary.
Let $G(\lambda,\mu)$ be an rv with density 
$y^{-\alpha-1}{\bf 1}_{\lambda<y\le \mu}/ \int_\lambda^\mu
 y^{-\alpha-1}\rmd y$, 
$H_\ell(\lambda,\mu)$ the sum of $\ell$ iid copies of 
$G(\lambda,\mu)$,  $\ell\ge 1$, $H_0(\lambda,\mu)=0$ 
and 
$c(\lambda,\mu)= (\lambda^{-\alpha} -\mu^{-\alpha})/\Gamma(1-\alpha)$.
 \begin{corollary}\label{3para}   
We have, for   $0<\lambda\le \mu\le 1$ and  $\ell=0,1,\ldots$, the limit
  \bea\label{f25}
 &&
      \lim_{n\to\infty}
    \PP\big(S_n^{\lambda,\mu}(\alpha,\theta) =\ell\big|K_n(\alpha,\theta) = \lf xn^\alpha\rf\big)\cr
    &&=
  \frac{  (xc(\lambda,\mu))^\ell e^{-xc(\lambda,\mu)}}
{ \ell! f_{Y_x^{(0)}}(1)}
 \times \frac{1}{2\pi }
 \int _{-\infty}^\infty e^{-\rmi \tau}
 \EE(e^{\rmi \tau Y_x^{(0)}})
 \EE(e^{\rmi \tau \calI(x,\lambda)})
  \EE(e^{\rmi \tau \calJ(x,\mu)})
\EE(e^{\rmi \tau H_\ell(\lambda,\mu)}) \rmd \tau.
  \eea
  The  righthand side of \eqref{f25} defines a proper probability mass function (adds to 1 over $\ell\ge 0$).
  \end{corollary}
  \noindent 
  Corollary \ref{3para} is proved in Section \ref{proofs} by inverting the mgf in \eqref{ST} for the nominated $f$.
 The integral in \eqref{f25}  
  has the value at 1
 of the density of the sum of independent copies of rvs
 $Y_x^{(0)}$, 
 $\calI(x,\lambda)$, $ \calJ(x,\mu)$, and 
 $H_\ell(\lambda,\mu)$.
  Again  the  righthand side of \eqref{f25}  does not depend on $\theta$. 
 The  characteristic functions  (cfs) of the $Y$, $\calI$ and $\calJ$  rvs are specified 
 in \eqref{f13},  \eqref{defI} and \eqref{defJ}.
   
 When $\mu=1$, \eqref{f25} gives the limiting conditional distribution of 
 $S_n^\lambda(\alpha,\theta)=\sum _{j=\lf \lambda n\rf}^{n} M_{jn}(\alpha,\theta)$ for 
 $0<\lambda<1$, 
 and  we  get our second main result --
 the marginal limiting distribution in \eqref{f28} --
  by integrating. 
For this, let  $G(\lambda)$ 
have density that of $G(\lambda,1)$, 
 let 
  $ c(\lambda) =(\lambda^{-\alpha} -1)/\Gamma(1-\alpha)$,
  and let  $ H_\ell(\lambda)$
 be the sum of  $\ell$ iid copies of  $G(\lambda)$, $0<\lambda<1$, $\ell\ge 1$,
 $ H_0(\lambda)=0$.
  Let 
 $(Y_x(\alpha, \lambda))_{x\ge 0}$ be a subordinator with
     L\'evy density $\alpha y^{-\alpha-1} {\bf 1}_{\{0<y\le \lambda\}}$,  independent of $H_\ell(\lambda)$, 
     and   let $f_{Y_x(\alpha,\lambda)+ H_\ell(\lambda)}(1)$ be
     the density function     of $Y_x(\alpha,\lambda)+ H_\ell(\lambda)$
      evaluated at 1.
The next result shows that  $S_n^\lambda(\alpha,\theta)$
(and similarly   $S_n^{\lambda,\mu}(\alpha,\theta)$, in Corollary \ref{3para}) 
have  limit distributions without centering or norming.

  \begin{theorem} \label{4para}
We have, for   $0<\lambda < 1$ and $\ell=0,1,\ldots$, the limit 
   \be\label{f28}
  \lim_{n\to\infty} \PP\big(S_n^\lambda(\alpha,\theta) =\ell\big)
= \frac{ 
(c(\lambda))^\ell\Gamma(\theta)
(\Gamma(1-\alpha))^{\theta/\alpha+\ell}
}
{\Gamma(\theta/\alpha) \ell!}
\int_0^\infty x^{\theta/\alpha+\ell-1}e^{-x\lambda^{-\alpha}}f_{Y_x(\alpha,\lambda)+ H_\ell(\lambda)}(1)\rmd x.
 \ee
  The right-hand side of \eqref{f28} 
defines a proper probability mass function.
    \end{theorem}

     There is an interesting connection of 
\eqref{f28} with the results of  \cite{ERGR2008}.   
Their paper concerns limit ``shapes'' for random structures on the set of partitions; that is, to do with the convergence of
 $S_n^\lambda(\alpha,\theta) $ (in our notation). See their Section 2  for a  sketch of the   historical development.
 They use a similar setup to our \eqref{N1}
     (Section 3 of  \cite{ERGR2008}) and   
 show  that measures of the type we are considering, in  what are termed the ``logarithmic'' and 
``convergent'' cases,  do not have limit shapes  (Corollary 3.1 of \cite{ERGR2008}), whereas for the ``expansive'' case, normed and centered versions of increments of $S_n^\lambda(\alpha,\theta) $ have limiting normal distributions
     (Theorem 4.1 of \cite{ERGR2008}).
     See also \cite{EG2024}, \cite{NikZ2013} and 
     \cite{Stufler2024}
     for related results.
Our results on  the  two-parameter  sampling formula are in the 
``convergent'' category and produce a quite different limiting regime.
We note also that the form of $S_n^{\lambda}(\alpha,\theta)$ suggests the possibility of a functional limit theorem. By choice of $\bff$ in \eqref{ST} we can obtain 
 finite-dimensional versions of  Corollary \ref{3para} and Theorem \ref{4para}, 
but we do not pursue this further here.

As another application of Theorem \ref{thm2N}, we give in the Appendix a quick derivation of Theorem 3 of  \cite{MallerShemehsavar2025}: 
the conditional asymptotic normality of a finite collection,  $(M_{jn}/K_n)_{1\le j\le J}$, after centering and an unusual $n^{\alpha/2}$ norming.

 \bigskip \noindent {\bf Remarks.}\  (i)\ 
Modifications of 
Condition \eqref{N1}  are appropriate for other similar models.
For example, \cite{Mohle2006a} introduces a model for coalescent
 processes with simultaneous multiple collisions.
See also  \cite{Mohle2006b}.
Similarly we can analyse the recursion in \cite{GIMS}.
Another important model is the  Poisson-Kingman $\PK^{(r)}(\rho)$ model. 
Given a Poisson point process with random intensity measure $\Gamma_r\rho$, where $r>0$,
$\Gamma_r$ is a Gamma rv and the 
 intensity  density $\rho(s)$ satisfies
\be\label{cond0}
\lim_{s \dto 0}\rho(s) = \infty, \quad 
\rho(s) <\infty \text{ for all } s> 0, \quad 
\text{and } 
\int_{0}^\infty(s\wedge 1) \rho(s) \rmd s<\infty,
\ee
 \cite{IpsenMaller2017} and  \cite{GIMS} 
 show how to construct the   $\PK^{(r)}(\rho)$ sampling model which generalises  the model $\PD_\alpha^{(r)}$
 of \cite{{IpsenMallerShemehsavar2020a}}
and a number of other models.
A formula for the distribution of the frequency spectrum  is
\be\label{dM}
P\big(\bfM_n=\bfm, K_n=k\big)
= n\int_{v>0}
\frac{r^{[k]}v^{n-1}}{\psi(v)^{r+k}}
	 \prod_{j=1}^{n}
	 \frac{1}{m_j!}
	\Big( \frac{\pi_{j}(v)}	 {j!}\Big)^{m_j} \rmd v,\ k\in \N_n,\
	\bfm\in A_{kn},
\ee 	
where $\psi(v)=1+\int_{0}^\infty \rho(x)(1- e^{-v x})\rmd x$ 
and $\pi_j(v)=(-1)^{j-1}\psi^{(j)}(v)$,
$v>0$.
In \cite{CZ2023} the  $\PD_\alpha^{(r)}$ model is derived in a point process context and applied to data on the relative velocities of galaxies. Simulations in \cite{CZ2023} show the utility of including values of $r>0$ in analyses.
 We note in \eqref{dM} under the integral sign a term analogous to the  righthand side of \eqref{N1} which suggests that the methods herein can be adapted to the  $\PK^{(r)}(\rho)$  class.
In fact  analyses along these lines have been done in \cite{IpsenMallerShemehsavar2020a} and \cite{MallerShemehsavar2025} for  $\PD_\alpha^{(r)}$.
A related formula is in \cite{Pitman2006}, p.81.
See also \cite{Mano2017} for the distributions of extreme sizes in Gibbs-type partitions.

(ii)\ 
    Many tests useful in genealogy   can be written as functions, often as  linear combinations, of the corresponding component spectrum; for example,  the Ewens–Watterson test statistic for homozygosity
\cite{Watterson1977}
 in our notation is 
$ \sum_{j=1}^n (j/n)^2 M_{jn}$.
 This is the expression in \eqref{STA} with $p=2$, so these results have direct application in genetics. 
We refer to \cite{MallerShemehsavar2025} and  \cite{ABT2003} 
 for  background, including discussion of the importance of the frequency spectrum 
 in population genetics.
See also \cite{YS2000},  \cite{Pitman2006}, \cite{GHP2007}, 
 \cite{HJL2021}, \cite{HJL2025},  \cite{James2008} 
 and \cite{James2025} 
for further discussion and references.
We mention that the ``large sample" in   \cite{James2008}  
refers to the behaviour of the model as  $\theta+n\alpha\to \infty$; 
see also his references.
 
 \section{Proofs}\label{proofs}\
\noindent{\bf  Proof of Theorem \ref{thm2N}:}\
 This proof  is carried out in two stages.
 In the first subsection we consider a fixed number $J$ of the $M_{jn}$.
   The general formula  is then
 applied in the second subsection to get the mgf in \eqref{ecenN}.
 
  \subsection{The first $J$ components of $\bfM_n$}\label{ss1}
   We proceed from \eqref{N1} to get a formula for the distribution of the first $J$ components of $\bfM_n$ by summing  over the remaining $n-J$ components. 
  Take $n\in \N$,  $k\in\N_n$, $1\le J<n$, $\bfm\in A_{kn}$, and write from \eqref{N1}
\begin{align}\label{M4.1N}
&
\PP(M_{jn}=m_j, 1\le j\le J,\, K_n=k)=
C_{nk}
\prod_{j=1}^J \frac{q_j^{m_j}}{m_j!} 
\sum_{\bfm^{(J)}\in A_{kn}^{(J)}}
\prod_{j=J+1}^n \frac{q_j^{m_j}}{m_j!},
\end{align}
 with $\bfm^{(J)}=(m_{J+1}, \ldots, m_n)$, 
\be\label{A2defN} 
A_{kn}^{(J)}=\Big\{m_j\ge 0, J< j\le n: \,
 \sum_{j=J+1}^nm_j=k^{(J)},\, 
\sum_{j=J+1}^n jm_j=n^{(J)}\Big\},
\ee
 $k^{(J)}=k-m_{+}^{(J)}$, $n^{(J)}=n-m_{++}^{(J)}$,  $m_{+}^{(J)}= \sum_{j=1}^Jm_j$ and $m_{++}^{(J)}= \sum_{j=1}^Jjm_j$.
Lemma \ref{LMu}  in the Appendix
 gives a key representation for summations such as in \eqref{M4.1N}. 
 To apply it in the present case, let
$\big(X_{in}^{(J)}\big)_{1\le i\le k^{(J)}}$ be i.i.d. integer valued rvs with distribution
\be\label{M8N}
\PP\big(X_{1n}^{(J)}=j\big) =
\frac{q_j}{\sum_{\ell=J+1}^n q_\ell}, \ J+1\le j\le n.
\ee
Then applying Lemma \ref{LMu} 
with $\mu_j=q_j$, 
 we can represent the  summation over $\bfm^{(J)}$ in \eqref{M4.1N}  
 as
\be\label{M9N}
\frac{1}{k^{(J)}!}
  \Big(\sum_{j=J+1}^n q_j \Big)^{ k^{(J)}}
\PP\Big(\sum_{i=1}^{ k^{(J)}} X_{in}^{(J)}= n^{(J)}\Big).
 \ee
 
 When $J=n$ this expression is interpreted as 1.
Recall $q_{+n}= \sum_{j=1}^n q_j$.
Let $q_{+J}= \sum_{j=1}^J q_j$ and set
\be\label{defqLN}
q_{+n}^{(J)}    
=\frac{\sum_{j=1}^J q_j}{\sum_{j=1}^n q_j}
=\frac{q_{+J}}{q_{+n}},   
\quad {\rm so\ that}\quad 
\sum_{j=J+1}^n q_j= (1- q_{+n}^{(J)})q_{+n}.
\ee
Using these, substitute  \eqref{M9N} into \eqref{M4.1N} to get
\be\label{M4.2N}
\PP\big(M_{jn}=  m_j, \,1\le j\le J,\,  K_n=k\big)  
=
\frac{C_{nk}(1-q_{+n}^{(J)})^{ k^{(J)}}}{ k^{(J)}!}
\times
\prod_{j=1}^J  \frac{q_{j}^{  m_j}}{  m_j!}
\times
(q_{+n})^{ k^{(J)}}
\PP\Big(\sum_{i=1}^{ k^{(J)}} X_{in}^{(J)}= n^{(J)}\Big),
\ee
where $n\in \N$,   $k\in \N_n$, $0<J<  n$, and 
  vectors $(    m_{1}, \ldots,   m_J):=    \wh\bfm^{(J)}$ are in
\be\label{A2defNw} 
\wh A_{kn}^{(J)}:=\Big\{   m_j\ge 0, 1< j\le J: \,
 \sum_{j=1}^J    m_j=m_{+}^{(J)},\, 
\sum_{j=1}^J j   m_j= m_{++}^{(J)}\Big\}.
\ee

 \subsection{Conditional Moment Generating Function
 of the Frequency Spectrum}\label{MGC}
Eq. \eqref{M4.2N} gives a representation for  the distribution of  the  subvector 
$ \bfM_n^{(J)}  =(M_{1n}, \ldots, M_{Jn})$ of $\bfM_n$.
  Correspondingly, let
  $M_{+n}^{(J)} =\sum_{j=1}^JM_{jn}$
  and  $M_{++n}^{(J)} =\sum_{j=1}^JjM_{jn}$.
  We include these variables  and $K_n$ 
  in defining the distribution of $\bfM_n^{(J)}$, and define 
also the event 
 \ben
 E_n(k,J)=
 E_n(k,m_+^{(J)},m_{++}^{(J)}) 
 :=
 \{K_n=k,\, M_{+n}^{(J)}=m_+^{(J)},\, M_{++n}^{(J)} =m_{++}^{(J)}\}.
 \een
  The second stage of the proof is to derive  a formula for the  mgf of  
 $\bfM_n^{(J)}/M_{+n}^{(J)}$ conditional on $E_n(k,J)$. 
 Let $\nu\in\R$,  $\bfu_J=(u_{1J},\ldots, u_{JJ}) \in \R^J$, 
and    
$q_{jJ}:=q_j/q_{+J}$.  
Then add over  $(m_{1}, \ldots,m_J)=    \wh\bfm^{(J)}$ in  \eqref{M4.2N} to get
 \bea\label{cfMJN}
&& \EE\Big(\exp\Big(\nu 
 \frac{\bfu_J^T\bfM_n^{(J)}}{ M_{+n}^{(J)}}\Big); E_n(k,J)
 \Big)
 = \sum_{ \wh\bfm^{(J)}\in \wh  A_{kn}^{(J)}} 
 \Big(\prod_{j=1}^Je^{\nu u_{jJ}   m_j/m_+^{(J)}}\Big)
 \times ({\rm the\ RHS\ of\ \eqref{M4.2N}})\cr
 &&\cr
&& =
\frac{C_{nk}(1-q_{+n}^{(J)})^{ k^{(J)}}
\big(q_{+J}\big)^{m_+^{(J)}}}{k^{(J)}!}
 \sum_{  \wh\bfm^{(J)}\in \wh  A_{kn}^{(J)}}
\prod_{j=1}^J \frac{1}{  m_j!} 
\big(q_{jJ}e^{ \nu u_{jJ}/m_+^{(J)}}  \big)^{ m_j}
\times
(q_{+n})^{ k^{(J)}}
 \PP\Big(\sum_{i=1}^{ k^{(J)}} X_{in}^{(J)}= n^{(J)}\Big),\cr&&
 \eea
 where  $k\in \N_n$,
 $\wh A_{kn}^{(J)}$
 is defined in \eqref{A2defNw}, and recall
 $q_{jJ}=q_j/q_{+J}$, 
$k^{(J)}=k-m_{+}^{(J)}$ 
and $n^{(J)}= n-m_{++}^{(J)}$. 

 We derived \eqref{cfMJN} 
 assuming $J<n$ but it remains true with $J=n$, when $m_{+}^{(n)}=k$, $m_{++}^{(n)}=n$  and $k^{(n)}=n^{(n)}=0$.
  The factors $(1-q_{+n}^{(J)})^{ k^{(J)}}$   and $ (q_{+n})^{ k^{(J)}}$ 
  and  the probabilities  involving  $X_{in}^{(n)}$    are then
 no longer present  (they are interpreted as being equal to 1).

Again apply Lemma \ref{LMu}, now with $J=0$ and $n$ set equal
to $J$, and
 $\mu_j = q_{jJ}e^{ \nu u_{jJ}/m_+^{(J)}}$,
 to 
 write the  summation over  $\wh\bfm^{(J)}$ 
 in \eqref{cfMJN} as
\be\label{ns0}
\frac{1}{m_+^{(J)}!}
 \PP\Big(\sum_{i=1}^{m_+^{(J)}} \wh X_{iJ}^{(\nu\bfu_J)}=m_{++}^{(J)}\Big)
\times
\Big(
\sum_{j=1}^J q_{j J}e^{ \nu u_{jJ}/m_+^{(J)}}  \Big)^{m_+^{(J)}},
\ee
which we can further rewrite as 
\be\label{ns}
\frac{1}{m_+^{(J)}!}
 \PP\Big(\sum_{i=1}^{m_+^{(J)}} \wh X_{iJ}^{(\nu\bfu_J)}=m_{++}^{(J)}\Big)
\times
\EE\Big(
\exp\Big(\frac{\nu} {m_+^{(J)}} \sum_{i=1}^{m_+^{(J)}} V_{iJ}^{(\bfu_J)} \Big)
\Big).
 \ee
 Here $\big(\wh X_{iJ}^{(\nu{\bfu_J}  )}\big)_{1\le i\le m_{+}^{(J)}}$ are i.i.d. with
\be\label{HM8N}
\PP\big(\wh X_{1J}^{(\nu{\bfu_J}  )}=j\big) 
=
\frac{q_{jJ}e^{ \nu u_{jJ}/m_+^{(J)}} }
{\sum_{\ell=1}^J q_{\ell J} e^{ \nu u_{\ell J}/m_+^{(J)}}},\ 1\le j\le J,
\ee
and   $(V_{iJ}^{(\bfu_J)})_{1\le i\le m_+^{(J)}}$  are 
 independent rvs,  each with distribution
 \be\label{dvND}
 P(V_{1J}^{(\bfu_J)}= u_{jJ})= q_{jJ} = \frac{q_j}{q_{+J}},\ 1\le j\le J,
 \ee
 so 
  $\sum_{i=1}^ {m_+^{(J)}}V_{iJ}^{(\bfu_J)}/m_+^{(J)}$ has mgf equal to the term $\Big(\sum_{j=1}^J q_{j J} 
  e^{ \nu u_{jJ}/m_+^{(J)}}
\Big)^{m_+^{(J)}}$
in \eqref{ns0}.
 Now \eqref{cfMJN} can be written as
  \bea\label{cN}
&& \EE\Big(
 \exp\Big(\nu 
 \frac{\bfu_J^T\bfM_n^{(J)}}{ M_{+n}^{(J)}}\Big); E_n(k,J)
 \Big) =
\frac{C_{nk}(1-q_{+n}^{(J)})^{ k^{(J)}}
\big(q_{+J}\big)^{m_+^{(J)}}}
{k^{(J)}! \, m_+^{(J)}!}
\times
 \PP\Big(\sum_{i=1}^{m_+^{(J)}} \wh X_{iJ}^{(\nu\bfu_J)}=m_{++}^{(J)}\Big)  \cr
  &&\hskip2cm
\times
(q_{+n})^{k^{(J)}}
 \PP\Big(\sum_{i=1}^{k^{(J)}} X_{in}^{(J)}= n^{(J)}\Big)
 \times
 \EE\Big(
\exp\Big(\frac{\nu} {m_+^{(J)}} \sum_{i=1}^{m_+^{(J)}} V_{iJ}^{(\bfu_J)}  \Big)
\Big).\qquad
 \eea
  \eqref{cN}  is a very general formula for the joint mgf.
 Setting $\nu=0$ in it  we obtain
   \bea\label{cN0}
 &&
 \PP\big( E_n(k,J)\big)
 =
 \PP\big(
 K_n=k, M_{+n}^{(J)}=m_+^{(J)}, M_{++n}^{(J)}=m_{++}^{(J)}
 \big)\cr
&&\cr
&& =
\frac{C_{nk}(1-q_{+n}^{(J)})^{ k^{(J)}}
\big(q_{+J}\big)^{m_+^{(J)}}}{k^{(J)}!\, m_+^{(J)}!}
\times
\PP\Big(\sum_{i=1}^{m_+^{(J)}} \wh X_{iJ}^{(0)}=m_{++}^{(J)}\Big)
\times
 (q_{+n})^{ k^{(J)}}
 \PP\Big(\sum_{i=1}^{k^{(J)}} X_{in}^{(J)}= n^{(J)}\Big),
 \eea
 where $X_{in}^{(J)}$ and  $ \wh X_{iJ}^{(0)}$ have the distributions in \eqref{M8N} and 
 \eqref{HM8N} (with $\nu=0$, thus, not  depending on $\nu$ or $\bfu_J$).
 
Again \eqref{cN} and \eqref{cN0}  
 remain true with $J=n$, when $m_{+}^{(n)}=k$, $m_{++}^{(n)}=n$  and $k^{(n)}=n^{(n)}=0$.
 Then 
 \eqref{cN0} becomes
\be\label{M4.2N1}
\PP\big( K_n=k\big)=
\frac{C_{nk}q_{+n}^k}{k!}
\PP\Big(\sum_{i=1}^{k} \wh X_{in}^{(0)}=n\Big),
\ee
where  $\big(\wh X_{in}^{(0)}\big)_{1\le i\le k}$ are i.i.d. with the  distribution in  \eqref{H0}.
Setting $J=n$ in \eqref{cN} gives
  \be\label{cN2}
 \EE\Big(
 \exp\Big(\nu 
 \frac{\bfu_n^T\bfM_n^{(n)}}{ M_{+n}^{(n)}}\Big);K_n=n \Big) =
\frac{C_{nk}\big(q_{+n}\big)^k}{k!}
 \PP\Big(\sum_{i=1}^{k} \wh X_{in}^{(\nu\bfu_n)}=n\Big) 
\times
 \EE\Big(
\exp\Big(\frac{\nu} {k} \sum_{i=1}^{k} V_{in}^{(\bfu_n)}  \Big)
\Big).
 \ee
With $\bfq_n$ as in \eqref{qnvec}, multiply both sides in \eqref{cN2} by
 $e^{-\nu\bfu_n^T\bfq_n}$
  to introduce centering
  and substitute for \eqref{M4.2N1} in this modified 
  \eqref{cN2}
  to get \eqref{ecenN}.
  The probability on the LHS of \eqref{M4.2N1} is positive, so the same is true of  the probability on the RHS, and the substitution implicit in deriving  \eqref{cN2} is well defined.
  This completes the proof of Theorem \ref{thm2N}.
  \halmos

Before considering specific asymptotic analyses,
 we set out the general strategy to be followed,  still using the general setup developed so far.

\medskip\noindent{\bf Asymptotic Analysis.}\
There are various ways of dealing with the large-sample ($n\to\infty)$ properties of the models.
  One possible strategy  is as follows.
  The summations in \eqref{ecenN} are of triangular arrays of  rvs that are i.i.d. in each row.
 Under  conditions to be specified,  the limiting behaviour of linear combinations of the elements of
    $\bfM_n$, conditional on $K_n=k_n$, can be deduced from   the limiting behaviour of  the normalised sums of i.i.d. rvs  in \eqref{ecenN} with $k=k_n$ 
    chosen appropriately.

    There are two types of limits to consider; those involving the $\wh X$, and those involving the $V$.
   Each set can be considered separately.
Note that the length of the vector $\bfu_n$ grows with $n$.

    \subsection{Asymptotics for the $\wh X$ variables}\label{Xvar}
To find the limits of the probabilities on the RHS of \eqref{ecenN}  we   require  a local limit theorem for the $\wh X$ variables.
A standard method is to use  the Fourier inversion  formula 
(\cite{GK1968}, p.233)  and write
 \be\label{3.num0}
n\PP\Big(\sum_{i=1}^{k} \wh X_{in}^{(\nu\bfu_n)}=n\Big)
=
 \frac{n}{2\pi } \int_{-\pi}^{\pi} e^{-\rmi \tau n}
\big( \phi_{\wh X_n}(\tau)\big) ^{  k} \rmd \tau
=
 \frac{1}{2\pi } \int_{-n\pi}^{n\pi} e^{-\rmi \tau}
\big( \phi_{\wh X_n}(\tau/n)\big) ^{  k} \rmd \tau,\ k\in \N_n,
 \ee
where  
$\phi_{\wh X_n}(\tau) := \EE\big(\exp(\rmi\tau   \wh X_{1n}^{(\nu\bfu_n)})\big)$, $\tau\in\R$.
Suppose that $k=k_n\to \infty$ 
 is such that, as $n\to\infty$, 
$
n^{-1} \sum_{i=1}^{k_n}  \wh X_{in}^{(\nu\bfu_n)}\todr Y^{(\nu,\bfu)}
$
for an infinitely divisible rv $Y^{(\nu,\bfu)}$ with density 
$f_{Y^{(\nu,\bfu)}}(y)$, $y>0$. 
 We then have convergence of the characteristic function (cf)
  $(\phi_{\wh X_n}(\tau/n)\big) ^{  k_n} \to 
 \EE (e^{\rmi \tau Y^{(\nu,\bfu)}})$
 and  via \eqref{3.num0} we might expect 
    \be\label{3.num1}
    \lim_{n\to\infty}
n\PP\Big(\sum_{i=1}^{k_n} \wh X_{in}^{(\nu\bfu_n)}=n\Big)
=
 \frac{1}{2\pi } \int_{-\infty}^{\infty} e^{-\rmi \tau}
 \EE (e^{\rmi \tau Y^{(\nu,\bfu)}}) \rmd \tau = f_{Y^{(\nu,\bfu)}} (1),
 \ee
 where the RHS is  the value of  the  density of   $Y^{(\nu,\bfu)}$ at the point 1.
 This will give the limit of $n$ times the numerator term in \eqref{ecenN}, and setting $\nu=0$ will give the limit of $n$ times the denominator,
 with $Y^{(0)}$ being the corresponding version of   $Y^{(\nu,\bfu)}$.
The main task is to justify taking the limit through the integral in \eqref{3.num0} to get \eqref{3.num1}.
 
 Supposing \eqref{3.num1} holds with $f_{Y^{(\nu,\bfu)}}(1)$ and $f_{Y^{(0)}}(1)$ finite and positive, 
 we then get from \eqref{ecenN},  as $n\to\infty$,
   \be\label{ece}
\EE\Big(
\exp\Big(\nu\bfu_n^T\Big(\frac{\bfM_n}{K_n}-\bfq_n\Big)\Big)\Big|K_n=k_n\Big)
\sim 
\EE\Big(\exp\Big(
  \frac{ \nu }{k_n} \sum_{i=1}^{k_n}
   \big(V_{in}^{(\bfu_n)}- \EE(V_{in}^{(\bfu_n)})\big)\Big)\Big)
   \times
\frac{f_{Y^{(\nu,\bfu)}}(1)}{f_{Y^{(0)}}(1)},  
   \ee
    for each $\nu$ and $\bfu_n$.
  Continuing with the $\wh X$ variables, to find the limiting distribution of the normed sum 
  $n^{-1} \sum_{i=1}^{k_n}  \wh X_{in}^{(\nu\bfu_n)}$
we can use   classical limit theorems for sums of  a triangular array, e.g., by verifying Conditions (i), (ii), (iii), of  Corollary 15.16, p.297, of \cite{Kallenberg2002}. 
 Kallenberg's Condition (i)  is
\be\label{210}
\lim_{n\to\infty}  k_n
\PP\big(\wh X_{1n}^{(\nu\bfu_n)}>\lf hn\rf\big) = 
\mu^{(\nu,\bfu)}\{(h,1]\}, \ 0<h\le 1,
\ee
for continuity points $h$ of a L\'evy measure $ \mu^{(\nu,\bfu)}$ which will attribute positive mass only to $(0,1]$; recall  that $\PP(0<\wh X_{1n}^{(\nu\bfu_n)}\le n)=1$.
For  Kallenberg's Condition (ii) we require
  \bea\label{10a0}
  &&
\lim_{n\to\infty} k_n
\EE\Big( n^{-1}\wh X_{1n}^{(\nu\bfu_n)} {\bf 1}_{\{\wh X_{1n}^{(\nu\bfu_n)}<hn\}}\Big)
=
  \int_{0<y\le h}y\mu^{(\nu,\bfu)}(\rmd y)
  =
  b- \int_{h<y\le 1}y\mu^{(\nu,\bfu)}(\rmd y),
\eea
where $b= \int_{0<y\le 1}y\mu^{(\nu,\bfu)}(\rmd y)$,
assumed finite. Kallenberg's Condition (iii) requires
  \be\label{10b0}
\lim_{n\to\infty} k_n
\EE\Big(( n^{-1}\wh X_{1n}^{(\nu\bfu_n)})^2 {\bf 1}_{\{\wh X_{1n}^{(\nu\bfu_n)}<hn\}}\Big)
=a+ \int_{0<y\le h}y^2\mu^{(\nu,\bfu)}(\rmd y)
\ee
for  a finite constant $a\ge 0$ and all $h>0$.

With these three conditions satisfied, 
    the normed sum 
 $n^{-1} \sum_{i=1}^{k_n} \wh X_{in}^{(\nu\bfu_n)}$ 
 converges in distribution to the infinitely divisible distribution 
 with characteristic exponent
 (see  Corollaries  15.16, 15.8, p.291, of \cite{Kallenberg2002})
 \bea\label{10c0}
 &&\rmi\tau  b -\tfrac{1}{2} \tau ^2 a +\int_{\R\setminus\{0\}}
 (e^{ \rmi\tau  y} -1 -  \rmi\tau  y{\bf 1}_{\{0<y<1}) \mu^{(\nu,\bfu)}(\rmd y)=
\int_{(0,1]}
 (e^{ \rmi\tau  y} -1 )\mu^{(\nu,\bfu)}(\rmd y)
 \eea
 assuming\footnote{If $a>0$ there would be a normal component to $Y^{(\nu,\bfu)}$.  This doesn't occur in our examples.}
  $a=0$
 and recalling $b=\int_{0<y<1} y\mu^{(\nu,\bfu)}(\rmd y)$.
The limit  distribution thus has  characteristic function   
\be\label{cfY}
\EE(e^{\rmi \tau Y^{(\nu,\bfu)}})=
\exp\Big(
\int_{(0,1]}
 (e^{ \rmi\tau  y} -1 ) \mu^{(\nu,\bfu)}(\rmd y)\Big).
 \ee
 An analogous approach can be used for the $X_{in}^{(0)}$ variables.

So far we have not specified 
$k_n$.
The choice of $k=k_n$ in \eqref{ecenN} will be dictated by the marginal limiting behaviour of $K_n$.
Next we turn to  the behaviour of the $V$-variables.

 \subsection{Asymptotics for the $V$ variables}\label{Vvar}
 The expectation on the right hand side of \eqref{ecenN} is the mgf of   the centered, normalised sum of independent rvs $V_{in}^{(\bfu_n)}$  
  having the mean and variance in \eqref{EV}.
  The matrix $\bfQ_n$ in \eqref{EV} 
  is a rank $n-1$ positive semi-definite matrix with the eigenvector ${\bf 1_n}=(1,1,\ldots, 1)\in \R^n$ corresponding to the zero eigenvalue; see \cite{TanSag1992}.
We impose the condition $\bfu_n^T\bfQ_n\bfu_n>0$, which requires $\bfu_n$ not to be a multiple of ${\bf 1_n}=(1,1,\ldots,1)$.
Note that when $\bfu_n={\bf 1_n}$,
$\bfu_n^T\bfM_n= \sum_{j=1}^nM_{nj}=K_n$.
In the Appendix 
we show as a check that \eqref{ecenN} reduces correctly in this case.

In \eqref{ece} we replace $\nu$ by $\nu b_n$ to introduce a norming sequence $b_n>0$.
Having given $k_n$,  
a possible choice of  $b_n$ is 
\be\label{bdef}
b_n= \sqrt{\frac{k_n}{ {\rm Var} (V_{1n}^{(\bfu_n)})}}
= \sqrt{\frac{k_n}{ \bfu_n^T\bfQ_n\bfu_n}},
\ee
thus, $k_n =b_n^2 \bfu_n^T\bfQ_n\bfu_n$
and  $b_n/k_n =1/\sqrt{k_n\bfu_n^T\bfQ_n\bfu_n}$.
  Then
   \be\label{nsgN}
   \frac{b_n}{k_n }
 \sum_{i=1}^{k_n} \big(V_{in}^{(\bfu_n)}- \EE(V_{in}^{(\bfu_n)}\big)
 =
  \frac{1}{\sqrt{k_n   \bfu_n^T\bfQ_n\bfu_n }}
 \sum_{i=1}^{k_n} \big(V_{in}^{(\bfu_n)}-
  \EE(V_{in}^{(\bfu_n)})\big).
 \ee
  The normalisation in \eqref{nsgN} is such that the variance of the expression therein is 1.
  An alternative choice for $b_n$ is  $b_n\sim ck_n$ for a constant $c>0$, which simplifies the expressions in \eqref{nsgN}.
  Having decided on $k_n$ and $b_n$ we can apply convergence theorems for  triangular arrays  or any equivalent method
 to investigate the  asymptotic behaviour   of the mgf on the RHS of \eqref{ece}.

  To sum up so far:
   the convergence  of     $\bfu_n ^T\bfM_n$, conditional on $K_n$, can be studied via that  of  normalised sums of iid rvs.
  We will verify that this procedure can be completed for the
  Pitman-Yor model.
 Following \cite{ABT2003}, Ch. 8, p.177, we analyse the convergence of functionals like $\sum_{j=1}^n f(j/n)M_{jn}$, for
 functions as specified in Section \ref{intro}; in particular, $f$ is bounded on $[0,1]$.
 In the previous notation, 
  we consider a weighting sequence of the form 
   $\bfu_n= (u_{jn}=f(j/n))_{1\le j\le n}$
   and let $\bff_n$ denote the vector $(f(j/n))_{1\le j\le n}$.
   For the Pitman-Yor  model  the  centering term in \eqref{ecenN} is not needed
   and we delete it from both sides.

\medskip\noindent{\bf Proof of Theorem \ref{2para0}:}\
For the 2-parameter model,  
refer to \eqref{FSFN} and  \eqref{qdef} and  assume \eqref{ig}.
Apply  \eqref{ecenN} with $\bfq_n$ deleted, $\bfu_n=\bff_n$, $k=k_n\in \N_n$, and $\nu$ replaced by $\nu b_n$, $b_n>0$, to  write
  \bea\label{TS-}
  &&
n^\alpha  \EE\Big(
\exp\Big(\frac{ \nu b_n\bff_n^T\bfM_n(\alpha,\theta)}{K_n(\alpha,\theta)}\Big); K_n(\alpha,\theta)=k_n\Big)\cr
   &&
   =
   \EE\Big(\exp\Big(
  \frac{ \nu b_n}{k_n} \sum_{i=1}^{k_n}
  V_{in}^{(\bff_n)}\Big)\Big)\times
\frac{\PP\Big(\sum_{i=1}^{k_n} \wh X_{in}^{(\nu\bff_n)}
=n\big)}
{ \PP\Big(\sum_{i=1}^{k_n} \wh X_{in}^{(0)}=n\Big)}
 \times 
 n^\alpha \PP(K_n(\alpha,\theta)=k_n).
\eea
Choose 
$k_n= \lf xn^\alpha\rf$, $x>0$, and  
notice that  \eqref{nsgN} simplifies considerably if we choose 
$b_n= k_n$, 
as we do now. 
Making these substitutions, the $\wh X$ and $V$ variables in 
\eqref{HM8Nn}--\eqref{dvN} and \eqref{TS-} then have distributions 
\be\label{HM8}
\PP\big(\wh X_{1n}^{(\nu f )}=j\big) 
=
\frac{q_{j}e^{ \nu  f(j/k_n)} }
{\sum_{\ell=1}^n q_{\ell } e^{ \nu  f(\ell/k_n)}}, \ {\rm where}\
q_j = \frac{\alpha \Gamma(j-\alpha)}{j!\Gamma(1-\alpha)}, 1\le j\le n
\ee
(we now write just $f$ for $\bff_n$)
with $\wh X_{1n}^{(0 )}$ obtained as the case $\nu=0$, 
and 
 \be\label{dvN2}
 P(V_{1n}^{(f)}= f(j/n))= q_{jn} =\frac{q_j}{\sum_{\ell=1}^n q_\ell},\ 1\le j\le n.
 \ee
 
 For the asymptotic analysis, we begin with the second factor on the RHS of  \eqref{TS-},  
and  verify \eqref{3.num1} for an appropriate version of $Y^{(\nu,f)}$.
For this we need the next lemma  whose proof is deferred to the Appendix.

 \begin{lemma}\label{ylem}
For the two-parameter  Pitman-Yor model  
 we have the following limits as $n\to\infty$.
 First, 
 \be\label{6.5a}
 n^{-1} \sum_{i=1}^{\lfloor x n^\alpha\rfloor} 
 \wh X_{in}^{(\nu f)}
 \todr Y_x^{(\nu,f)},
 \ee
where the $ \wh X_{in}^{(\nu f)}$ are i.i.d. with the distribution in \eqref{HM8} and $Y_x^{(\nu,f)}$ is  infinitely divisible with 
 characteristic function
\be\label{ycf}
\EE\big(e^{\rmi \tau Y_x^{(\nu,f)}}\big)
=
\exp\Big(\frac{\alpha x}{\Gamma(1-\alpha)}
  \int_{0}^1
 (e^{ \rmi\tau  y} -1 ) e^{\nu f(y)} y^{-\alpha-1} \rmd y\Big). 
\ee
This integral  is absolutely convergent, so each rv $Y_x^{(\nu,f)}$ has a continuous density.  Each $Y_x^{(\nu,f)}$  has support $(0,\infty)$.
A second limit is 
    \be\label{3.n}
    \lim_{n\to\infty}
n\PP\Big(\sum_{i=1}^{\lf xn^\alpha\rf} \wh X_{in}^{(\nu f)}=n\Big)
=
 \frac{1}{2\pi } \int_{-\infty}^{\infty} e^{-\rmi \tau}
 \EE (e^{\rmi \tau Y_x^{(\nu,f)}}) \rmd \tau = f_{Y_x^{(\nu,f)}} (1),
 \ee
 where the RHS 
 is finite  and positive for each $\nu$, $f$ and $x$.
 \end{lemma}
\eqref{3.n} applied to the ratio of probabilities in  \eqref{TS-} gives $ f_{Y_x^{(\nu,f)}}(1)/ f_{Y_x^{(0)}}(1)$, 
as required in \eqref{ST}.
Next we analyse the first factor on the RHS of  \eqref{TS-},  concerning the $V$ variables.
At this stage  \eqref{ig} is needed.
The proof of the next lemma is also deferred to the Appendix.

  \begin{lemma}\label{vcon}
 Assume \eqref{ig} and the above choices of $\bfu_n$, $b_n$ and $k_n$. 
The limit of the mgf on   the RHS of   \eqref{TS-}
 is then
     \be\label{vfin}
      \exp\Big(\frac{x\alpha}{\Gamma(1-\alpha)}
   \int_0^1 y^{-\alpha-1}  \big(
e ^{\nu  f(y) }  -1\big)\rmd y\Big).
\ee
\end{lemma}

For  the third factor on the RHS of  \eqref{TS-}, use  \eqref{M4.2N1} and the formula  for  $C_{nk}(\alpha, \theta)$
from  \eqref{FSFN}
to write, for any $k_n\to\infty$, 
\bea\label{3}
\PP(K_n(\alpha, \theta)=k_n)
&=&
\frac{C_{nk_n}(\alpha, \theta)q_{+n}^{k_n}}{k_n!}
\times
\PP\Big(\sum_{i=1}^{k_n} \wh X_{in}^{(0)}=n\Big)\cr
&=&
\frac{C_{nk_n}(\alpha, \theta)}{C_{nk_n}(\alpha,0)}
\times
\PP(K_n(\alpha,0)=k_n)
\cr
&=&
\frac{\Gamma(\theta+1)}{\Gamma(\theta/\alpha+1)}
\times
\frac{\Gamma(n)\Gamma(\theta/\alpha+k_n)}
{\Gamma(n+\theta)\Gamma(k_n)}
\times
 \PP(K_n(\alpha,0)=k_n)\cr
 &\sim&
 \frac{\Gamma(\theta+1)}{\Gamma(\theta/\alpha+1)}
\times \frac{k_n^{\theta/\alpha}}{n^\theta} 
\times \PP(K_n(\alpha,0)=k_n),
\ {\rm as}\ n\to\infty.
\eea
 In the last line we used Stirling's formula to  approximate the second factor.
By \eqref{M4.2N1} and  $C_{nk}(\alpha,0)= n\Gamma(k)/\alpha$, we have 
\be\label{5}
n^\alpha \PP(K_n(\alpha,0)=k_n)
= \frac{n^\alpha}{\alpha k_n}\times q_{+n}^{k_n}
\times n\PP\Big(\sum_{i=1}^{k_n} \wh X_{in}^{(0)}=n\Big).
\ee
Substitute $k_n= \lf xn^\alpha\rf$ in this, 
and note that, as $n\to\infty$, we have, successively,  $n^\alpha/\alpha k_n\to x^{-1}/\alpha$, 
\be\label{5a}
q_{+n}^{k_n}
=\Big(\sum_{j=1}^n 
\frac{\alpha \Gamma(j-\alpha)}{j!\Gamma(1-\alpha)}
\Big)^{\lf xn^\alpha\rf}
=\Big(1-\frac{\alpha}{\Gamma(1-\alpha)}
\sum_{j>n} 
\frac{\Gamma(j-\alpha)}{j!}
\Big)^{\lf xn^\alpha\rf}
\to e^{-x/\Gamma(1-\alpha)}
\ee
(because $\Gamma(j-\alpha)/j! \sim j^{-1-\alpha}$ as $j\to \infty$),
and 
$n\PP\Big(\sum_{i=1}^{k_n} \wh X_{in}^{(0)}=n\Big)\to f_{Y_x^{(0)}}(1)$  (by \eqref{3.n})   to  conclude
\be\label{4}
\lim_{n\to\infty}  n^\alpha\PP(K_n(\alpha,0)=\lf xn^\alpha\rf)
= x^{-1} e^{-x/\Gamma(1-\alpha)} f_{Y_x^{(0)}}(1)/\alpha.
\ee
The RHS of \eqref{4} equals $g_\alpha(x)$, where 
$g_\alpha(x)$ is the Mittag-Leffler density in \cite{{Pitman2006}}, p.11.
This fact, which we will use in the proof of \eqref{f25}, 
 is shown in Theorem 3 of  \cite{MallerShemehsavar2025}.
Consequently from \eqref{3}  
and \eqref{4} we get 
\be\label{3b}
\lim_{n\to\infty} n^\alpha
\PP(K_n(\alpha,\theta)=\lf xn^\alpha\rf)
=
 \frac{x^{\theta/\alpha-1}
\Gamma(\theta+1)}{\alpha\Gamma(\theta/\alpha+1)}
e^{-x/\Gamma(1-\alpha)} 
 f_{Y_x^{(0)}}(1)
 =
 \frac{\alpha x^{\theta/\alpha}
\Gamma(\theta)}{\Gamma(\theta/\alpha)}g_\alpha(x).
\ee
\eqref{3b} contains the local limit theorem in \cite{Pitman1999} and 
\cite{{Pitman2006}}, p.72, which is a component of \eqref{ST}.
Pitman's proof is via ladder times of random walks, our  proof  uses Lemma \ref{ylem}.

The limit of the RHS of \eqref{TS-} is given by multiplying  \eqref{vfin} with the limit of the ratio of probabilities 
on the RHS of \eqref{TS-}, using \eqref{3.n}, together with the RHS of \eqref{3b}.   This gives \eqref{ST}. The equality in  \eqref{f1} follows from \eqref{ycf} and the  Fourier inversion  formula in 
\cite{GK1968}, p.233.

 The integral in \eqref{f1} is absolutely convergent,  as shown in the proof of  Lemma \ref{ylem} (see \eqref{a1}), so $Y_x^{(\nu,f)}$ has a continuous density $f_{Y_x^{(\nu,f)}}(z)$ for all $z>0$  by Prop. 28.1, p.190, of \cite{sato99}.
Since $Y_x^{(\nu,f)}$ has zero drift and L\'evy measure with support in $[0,\infty)$,
$Y_x^{(\nu,f)}$ has support $[0,\infty)$
 by Thm. 24.3 and Cor. 24.8, p.151, of \cite{sato99}.
 Finally, the factor $f_{Y_\theta^{(\nu,f)}}(1)$   in \eqref{ST} is positive (including, for $\nu=0$), because 
from  Theorem 2 of \cite{HT1975} we can deduce that 
  the density of the absolutely continuous infinitely divisible rv
   $Y_x^{(\nu,f)}$  
is positive a.e. w.r.t Lebesque measure over its support.
But since the density is continuous, 
$f_{Y_x^{(\nu,f)}}>0$ on $(0,\infty)$, 
and in particular,  $f_{Y_x^{(\nu,f)}}(1)>0$.
This completes the proof of Theorem \ref{2para0}.      \halmos

The next result, a consequence of \eqref{3b},  is also proved in the Appendix.

     \begin{lemma}\label{blem}
     In the Pitman-Yor  model, we have
           \be\label{f29}
1=  \lim_{n\to\infty}
n^\alpha\int_0^\infty  \PP\big(K_n(\alpha,\theta)= \lf xn^\alpha\rf  \big)\rmd x
= \frac{\Gamma(\theta)}
{\Gamma(\theta/\alpha)}
\int_0^\infty x^{\theta/\alpha-1}e^{-x/\Gamma(1-\alpha)}f_{Y_x^{(0)}}(1)\rmd x.
\ee
      \end{lemma}

For the remaining results it's convenient to note the following conditional version of \eqref{ST},
whose proof is just to divide the RHS of \eqref{ST} by the RHS of \eqref{3b}.
  When \eqref{ig}  holds,  we have for  $ \nu\in\R$ and $x>0$:
  \be\label{ST2}
\lim_{n\to\infty}  
  \EE  \Big( \exp\Big(\nu  
\bff_n^T \bfM_{n}(\alpha,\theta)\Big)\Big|K_n(\alpha,\theta) =\lfloor x n^\alpha\rfloor \Big)
=
\frac{f_{Y_x^{(\nu,f)}}(1)}{f_{Y_x^{(0)}}(1)} 
   \times
      \exp\Big(\frac{x\alpha}{\Gamma(1-\alpha)}
   \int_0^1 y^{-\alpha-1}  \big(
e ^{\nu  f(y) }  -1\big)\rmd y\Big).
 \ee

 \medskip\noindent{\bf Proof of Corollary \ref{3parc}}:\
\eqref{STA}  follows immediately from  \eqref{ST2} with $f(y)=y^p$;  $p>\alpha$ is required for \eqref{ig}.
\halmos

  The case $p=1$ in Corollary \ref{3parc} provides a check on the calculations since we know $\sum_{j=1}^n jM_{jn}=n$.
In the Appendix 
we show that the corollary reduces correctly in this case.

    \medskip  \noindent  {\bf Proof of   Corollary \ref{3para}:}\
Working from \eqref{ST2}, substitute in \eqref{f1}  the function 
  $f(y)={\bf 1}_{\lambda\le y\le \mu}$ 
 to get
  \be\label{f11}
f_{Y_x^{(\nu,f)}}(1) = \frac{1}{2\pi}\int _{-\infty}^\infty e^{-\rmi \tau}
  \exp\Big(\frac{x\alpha }{\Gamma(1-\alpha)}
\Big[ \int_{0}^\lambda   +e^{\nu  }   \int_\lambda^\mu 
+\int_\mu^1\Big]
  (e^{ \rmi\tau  y} -1 ) y^{-\alpha-1} \rmd y
 \Big) \rmd \tau.
 \ee
 Write 
$  \int_{0}^\lambda   +e^{\nu  }   \int_\lambda^\mu 
+\int_\mu^1
  =
  \int_{0}^1   +(e^{\nu  } -1)  \int_\lambda^\mu$,
then
  \be\label{f12}
 f_{Y_x^{(\nu,f)}}(1) 
= \frac{1}{2\pi}\int _{-\infty}^\infty e^{-\rmi \tau}
  \EE(e^{\rmi \tau Y_x^{(0)}})
  \exp\Big(\frac{x\alpha (e^{\nu  }-1) }{\Gamma(1-\alpha)}
  \int_\lambda^\mu
  (e^{ \rmi\tau  y} -1 ) y^{-\alpha-1} \rmd y
 \Big) \rmd \tau,
 \ee
 where by \eqref{ycf} 
 \be\label{f13}
   \EE(e^{\rmi \tau Y_x^{(0)}})
   =
    \exp\Big(\frac{x\alpha}{\Gamma(1-\alpha)}
 \int_0^1
  (e^{ \rmi\tau  y} -1 ) y^{-\alpha-1} \rmd y\Big).
 \ee
From \eqref{ST2} with $\bff_n^T \bfM_{n}(\alpha,\theta)=
\sum _{j=\lf \lambda n\rf}^{\lf \mu n\rf} M_{jn}(\alpha,\theta)=S_n^{\lambda,\mu}(\alpha,\theta)$
and \eqref{f12} 
we can write
    \bea\label{f15}
    &&
    \lim_{n\to\infty}
    \EE\big(\exp(\nu S_n^{\lambda,\mu}(\alpha,\theta) |K_n(\alpha,\theta)= \lf xn^\alpha\rf\big)
\times f_{Y_x^{(0)}}(1)  \cr 
 &=&
f_{Y_x^{(\nu,f)}}(1) 
   \times
      \exp\Big(\frac{x\alpha}{\Gamma(1-\alpha)}
         \int_0^1 y^{-\alpha-1}  \big(
e ^{\nu  f(y) }  -1\big)\rmd y\Big)
\cr
&=&
f_{Y_x^{(\nu,f)}}(1) 
   \exp\Big(\frac{x\alpha (e^\nu-1)}{\Gamma(1-\alpha)}
    \int_\lambda^\mu  y^{-\alpha-1} \rmd y \Big)\cr
    &=&
 \frac{1}{2\pi}\int _{-\infty}^\infty e^{-\rmi \tau}
  \EE(e^{\rmi \tau Y_x^{(0)}})
  \exp\Big(\frac{x\alpha }{\Gamma(1-\alpha)}
(e^{\nu }-1)   \int_\lambda^\mu
 e^{ \rmi\tau  y}  y^{-\alpha-1} \rmd y
 \Big) \rmd \tau.
 \eea
 
To simplify this define 
 \be\label{f17a}
 g(x, \rmi\tau,\lambda,\mu)
 =
 \frac{x\alpha }{\Gamma(1-\alpha)}
  \int_\lambda^\mu  e^{ \rmi\tau  y}  y^{-\alpha-1} \rmd y
  =   xc(\lambda,\mu) 
 \frac{  \int_\lambda^\mu e^{ \rmi\tau  y}  y^{-\alpha-1} \rmd y}
 {  \int_\lambda^\mu  y^{-\alpha-1} \rmd y}
 \ee
   where 
 \be\label{f17}
   c(\lambda,\mu) 
  := \frac{\alpha }{\Gamma(1-\alpha)}
    \int_\lambda^\mu  y^{-\alpha-1} \rmd y
    =
     \frac{\lambda^{-\alpha} -\mu^{-\alpha}}{\Gamma(1-\alpha)}.
 \ee
 With this notation the expression on the RHS of \eqref{f15} equals
 \be\label{f18}
   \frac{1}{2\pi}\int _{-\infty}^\infty e^{-\rmi \tau}
  \EE(e^{\rmi \tau Y_x^{(0)}})
  \exp\big( g(x, \rmi\tau,\lambda, \mu)
(e^{\nu  }-1) \big)
\rmd \tau.
 \ee
 
Now we invert the mgfs in \eqref{f15} and \eqref{f18}.
Taking our cue from the Poisson distribution, define
 \be\label{f18a}
 P_{gs} =P_{gs}(x, \rmi\tau,\lambda, \mu)
 = e^{- g(x, \rmi\tau,\lambda, \mu)}
\frac{ \big( g(x, \rmi\tau,\lambda, \mu)\big)^s}
{s!}, \ s=0,1,\ldots.
 \ee
 The expression in \eqref{f18} can then be written as 
 \be\label{f19}
   \frac{1}{2\pi}\int _{-\infty}^\infty e^{-\rmi \tau}
  \EE(e^{\rmi \tau Y_x^{(0)}})
  \sum_{s=0}^\infty e^{\nu  s} 
P_{gs}(x, \rmi\tau,\lambda, \mu)
\rmd \tau,
 \ee
 and by \eqref{f15} this equals
 \be\label{f20}
     \lim_{n\to\infty}
    \EE\big(\exp\big(\nu S_n^{\lambda, \mu} (\alpha,\theta) \big)\big|K_n(\alpha,\theta) = \lf xn^\alpha\rf\big)
 \times
 f_{Y_x^{(0)}}(1).
 \ee
 Replace $\nu$ by $-\nu$, $\nu>0$, 
 so the mgf in \eqref{f20} becomes a (conditional) Laplace transform, $\phi_n(\nu)$.  Apply the inversion operator in \cite{Feller}, p.230,
 to write
 \be\label{f22}
    \PP\big(S_n^{\lambda, \mu}(\alpha,\theta) \le z\big|K_n(\alpha,\theta) = \lf xn^\alpha\rf\big)
    =
    \lim_{L\to\infty}\sum _{\ell\le zL} \frac{(-1)^\ell}{\ell!}
    L^\ell \phi_n^{(\ell)}(L),\ z>0,
 \ee
 at points of continuity of the distribution.
 Applying the same operator to 
 $  \sum_{s=0}^\infty e^{-\nu  s} 
P_{gs}(x, \rmi\tau,\lambda, \mu)$
(from \eqref{f19} with $\nu$ replaced by $-\nu$) produces
 \ben
 \sum_{s\le z} P_{gs}(x, \rmi\tau,\lambda, \mu)
 =
  \sum_{s\le z} e^{- g(x, \rmi\tau,\lambda, \mu)}
\frac{ \big( g(x, \rmi\tau,\lambda, \mu)\big)^s}{s!}
 \een
  at points of continuity.
Interchange limit and integral in \eqref{f19}   to obtain
 \bea\label{f23}
&& \lim_{n\to\infty}
  \PP\big(S_n^{\lambda, \mu}(\alpha,\theta) \le z\big|K_n(\alpha,\theta) = \lf xn^\alpha\rf\big)\cr
&&
    =  \frac{1}{2\pi f_{Y_x^{(0)}}(1)}\int _{-\infty}^\infty e^{-\rmi \tau}
  \EE(e^{\rmi \tau Y_x^{(0)}})
    \sum_{s\le z} e^{- g(x, \rmi\tau,\lambda, \mu)}
\frac{ \big( g(x, \rmi\tau,\lambda, \mu)\big)^s}{s!} \rmd \tau.
 \eea
 Replacing $z$ by $\ell$ then subtracting the same expressions for $\ell-1$ gives, for $\ell=0,1,\ldots$,
  \bea\label{f24}
  &&
      \lim_{n\to\infty}
    \PP\big(S_n^{\lambda, \mu}(\alpha,\theta) =\ell\big|K_n(\alpha,\theta) = \lf xn^\alpha\rf\big) \cr
    &&
    =  \frac{1}{2\pi  f_{Y_x^{(0)}}(1)}\int _{-\infty}^\infty e^{-\rmi \tau}
  \EE(e^{\rmi \tau Y_x^{(0)}})
e^{- g(x, \rmi\tau,\lambda, \mu)}
\frac{ \big( g(x, \rmi\tau,\lambda, \mu)\big)^\ell}{\ell!} \rmd \tau.
 \eea
Take $\nu=0$ in \eqref{f15} to see that
\be\label{to1}
  \frac{1}{2\pi  f_{Y_x^{(0)}}(1)}
\int _{-\infty}^\infty e^{-\rmi \tau}
  \EE(e^{\rmi \tau Y_x^{(0)}}) \rmd \tau =1,
\ee
so adding the RHS of \eqref{f24} over $\ell\ge 0$ gives 1,  the limit is a proper distribution, and  the interchange  is valid.

To rewrite \eqref{f24} in the  form of \eqref{f25},  notice from \eqref{f17a} that
\ben
e^{- g(x, \rmi\tau,\lambda, \mu)}
=
\exp\Big(-\frac{\alpha x}{\Gamma(1-\alpha)}
 \int_\lambda^\mu e^{ \rmi\tau  y} y^{-\alpha-1} \rmd y\Big)
 = 
e^{- x c(\lambda,\mu)}
 \EE(e^{\rmi \tau \calI(x,\lambda)})
  \EE(e^{\rmi \tau \calJ(x,\mu)}),
\een 
where 
$\calI(x,\lambda)$ and $\calJ(x,\mu)$ are infinitely  divisible rvs with 
\be\label{defI}
 \EE(e^{\rmi \tau \calI(x,\lambda)})
 =
\exp\Big(\frac{\alpha x}{\Gamma(1-\alpha)}
 \int_0^\lambda
(e^{ \rmi\tau  y}-1)   y^{-\alpha-1} \rmd y\Big)
\ee
and 
\be\label{defJ}
 \EE(e^{\rmi \tau \calJ(x,\mu)})
 =
\exp\Big(-\frac{\alpha x}{\Gamma(1-\alpha)}
 \int_0^\mu
(e^{ \rmi\tau  y}-1)   y^{-\alpha-1} \rmd y\Big).
\ee
Define  also a rv $G(\lambda,\mu)$ with density 
$y^{-\alpha-1}{\bf 1}_{\lambda<y\le \mu}/ \int_\lambda^\mu
 y^{-\alpha-1} \rmd y$.
Then from \eqref{f17a} and\eqref{f17}
\be\label{f26g}
  g(x, \rmi\tau,\lambda,\mu)=
 x c(\lambda,\mu)
  \EE(e^{\rmi \tau G(\lambda,\mu)}) 
  \quad  {\rm and}  \quad
\big( g(x, \rmi\tau,\lambda,\mu)\big)^\ell
 =  (xc(\lambda,\mu)^\ell
 \EE(e^{\rmi \tau \sum_{s=1}^\ell G_s(\lambda,\mu)}),
\ee
where $ G_s(\lambda,\mu)$, $1\le s\le \ell$, are iid copies of 
 $G(\lambda,\mu)$.
Then let $ H_\ell(\lambda,\mu)=  \sum_{s=1}^\ell G_s(\lambda,\mu)$
 be the sum of  $\ell$ iid copies of  $G(\lambda,\mu)$ and rewrite the RHS of \eqref{f24}  to obtain \eqref{f25}.   \halmos

    \medskip  \noindent  {\bf Proof of  Theorem \ref{4para}:}\
    For \eqref{f28}, start from 
       \be\label{f26b}
n^\alpha  \PP\big(S_n^\lambda(\alpha,\theta) =\ell;K_n(\alpha,\theta)= \lf xn^\alpha\rf\big)
    =
    \PP\big(S_n^\lambda(\alpha,\theta) =\ell\big| K_n(\alpha,\theta) = \lf xn^\alpha\rf\big)
    \times
 n^\alpha \PP\big(K_n(\alpha,\theta)= \lf xn^\alpha\rf\big),
 \ee
and multiply \eqref{f24} and the RHS of \eqref{3b} together to get the limit of the LHS of \eqref{f26b} equal to 
 \be\label{26d}
  \frac{1}{2\pi  f_{Y_x^{(0)}}(1)}\int _{-\infty}^\infty e^{-\rmi \tau}
  \EE(e^{\rmi \tau Y_x^{(0)}})
e^{- g(x, \rmi\tau,\lambda, 1)}
\frac{ \big( g(x, \rmi\tau,\lambda, 1)\big)^\ell}{\ell!} \rmd \tau
\times
  \frac{x^{\theta/\alpha-1}
\Gamma(\theta)}{\Gamma(\theta/\alpha)}
e^{-x/\Gamma(1-\alpha)}  f_{Y_x^{(0)}}(1).
 \ee
Cancel the factors $f_{Y_x^{(0)}}(1)$ and use \eqref{f26g}  with $\mu=1$,  $G(\lambda)=G(\lambda,1)$, $H(\lambda)=H(\lambda,1)$ 
 and 
$c(\lambda)=c(\lambda,1)=(\lambda^{-\alpha}-1)/\Gamma(1-\alpha)$, to write \eqref{26d} as
 \be\label{27}
 \frac{ (xc(\lambda))^\ell x^{\theta/\alpha-1}\Gamma(\theta) e^{-x/\Gamma(1-\alpha)}  }{2\pi\ell!\Gamma(\theta/\alpha)}
\int _{-\infty}^\infty e^{-\rmi \tau}
\EE(e^{\rmi \tau Y_x^{(0)}})
e^{- g(x, \rmi\tau,\lambda, 1)}
 \EE(e^{\rmi \tau H_\ell(\lambda)})
\rmd \tau.
 \ee
 Substituting from \eqref{f13} and \eqref{f17a} gives
  \bea\label{28}
  \EE(e^{\rmi \tau Y_x^{(0)}})
e^{- g(x, \rmi\tau,\lambda, 1)}
&=&
    \exp\Big(\frac{x\alpha}{\Gamma(1-\alpha)}
    \Big[
 \int_0^1  (e^{ \rmi\tau  y} -1 ) 
 -
  \int_\lambda^1  e^{ \rmi\tau  y}\Big]
 y^{-\alpha-1} \rmd y\Big)\cr
 &=&
     \exp\Big(\frac{\alpha x}{\Gamma(1-\alpha)}
 \int_0^\lambda  (e^{ \rmi\tau  y} -1 ) 
 y^{-\alpha-1} \rmd y\Big)
 \times
  \exp\Big(-\frac{x(\lambda^{-\alpha}-1)}{\Gamma(1-\alpha)}\Big)\cr
  &=&
    \exp\Big(\frac{x}{\Gamma(1-\alpha)}\Big)
      \exp\Big(-\frac{x\lambda^{-\alpha}}{\Gamma(1-\alpha)}\Big)
 \exp\Big(\frac{\alpha x}{\Gamma(1-\alpha)}
 \int_0^\lambda  (e^{ \rmi\tau  y} -1 ) 
 y^{-\alpha-1} \rmd y\Big).\qquad \
  \eea
Substitute this  in \eqref{27} and integrate over $0<x<\infty$, changing variable from $x/\Gamma(1-\alpha)$ to $x$, to get an expression for the required marginal pdf as
  \bea\label{29}
  &&
 \frac{(c(\lambda))^\ell\Gamma(\theta) 
 (\Gamma(1-\alpha))^{\theta/\alpha+\ell}   }{\Gamma(\theta/\alpha)2\pi\ell!}
 \int_0^\infty 
 x^{\theta/\alpha+\ell-1}  e^{-x\lambda^{-\alpha}}\times
 \cr
 &&
 \hskip 3cm 
 \times
\int _{-\infty}^\infty e^{-\rmi \tau}
    \exp\Big(\alpha x
 \int_0^\lambda  (e^{ \rmi\tau  y} -1 ) 
 y^{-\alpha-1} \rmd y\Big)
  \EE(e^{\rmi \tau H_\ell(\lambda)})\rmd x 
\rmd \tau.
 \eea
 The inner integral is absolutely convergent because the function
  \bea\label{a2}
 \Big|  \exp\Big(\alpha x
 \int_0^\lambda  (e^{ \rmi\tau  y} -1 ) 
 y^{-\alpha-1} \rmd y\Big)\Big|
 &=&
 \exp\Big(-\alpha x
  \int_{0}^\lambda
 (1-\cos(\tau  y) )y^{-\alpha-1} \rmd y\Big)\cr
 &=&
  \exp\Big(-\alpha x|\tau|^{\alpha}
  \int_{0}^{|\tau|}(1-\cos y)y^{-\alpha-1} \rmd y\Big)
 \eea
is absolutely integrable for $\tau\in\R$. 
 The integral  (divided by $2\pi$) equals
 \be\label{30a}
 \frac{1}{2\pi} \int _{-\infty}^\infty e^{-\rmi \tau}
 \EE(e^{\rmi \tau Y_x^{(\alpha,\lambda)}})
  \EE(e^{\rmi \tau H_\ell(\lambda)})
  \rmd \tau
  =
  f_{Y_x(\alpha,\lambda)+ H_\ell(\lambda)}(1)\rmd x,
 \ee
where $ Y_x^{(\alpha,\lambda)}$ is the subordinator in  \eqref{f28}. 
 Substituting \eqref{30a} in \eqref{29} we get \eqref{f28}. 
 
   It remains to show that the RHS of \eqref{f28} defines a proper probability mass function. We defer this demonstration till after the proofs of the lemmas in Section \ref{app}.
 \halmos

 \bigskip \noindent {\bf Remarks.}\  (i)\
 \eqref{f28} expresses the limiting distribution as a real variable integral.
 An alternative representation as a complex integral is useful too.
Let
 \be\label{hdef}
 h(\alpha,\tau,\lambda)
 =
 \lambda^{-\alpha}-\alpha 
\int_0^\lambda \big(e^{ \rmi\tau  y}-1\big) y^{-\alpha-1} \rmd y,
\ {\rm  for}\   \lambda\in(0,1), \tau\in\R,
 \ee
 and write the RHS of \eqref{28} as 
 $e^{x/\Gamma(1-\alpha)}   e^{-x h(\alpha,\tau,\lambda)/\Gamma(1-\alpha)}$.
 Substitute this in \eqref{27} and  replace the term 
 $(c(\lambda))^\ell \EE(e^{\rmi \tau H_\ell(\lambda)})$ 
 (see  \eqref{f17} and \eqref{f26g}) 
 with 
 $\big(\alpha \int_\lambda^1 e^{ \rmi\tau  y} y^{-\alpha-1} \rmd y/\Gamma(1-\alpha)\big)^\ell$.  
Then change variable from $x/\Gamma(1-\alpha)$ to $x$  in \eqref{27}, 
 and change the order of integration, to obtain the following alternative to \eqref{f28}:
 \bea\label{51}
 &&  
 \lim_{n\to\infty} \PP\big(S_n^\lambda(\alpha,\theta) =\ell\big)\cr
 &&
 =
 \frac{ \Gamma(\theta)
( \Gamma(1-\alpha))^{\theta/\alpha}}
 { \Gamma(\theta/\alpha)2\pi\ell!}
 \int _{-\infty}^\infty  e^{-\rmi \tau}
 \int_0^\infty 
 \Big(\alpha \int_\lambda^1 e^{ \rmi\tau  y} y^{-\alpha-1} \rmd y\Big)^\ell
x^{\theta/\alpha+\ell-1}
e^{-x h(\alpha,\tau,\lambda)}
\rmd x
\rmd \tau\cr
&&
= 
\frac{\Gamma(\theta)
\Gamma(\theta/\alpha+\ell)
(\Gamma(1-\alpha))^{\theta/\alpha}}
{ \Gamma(\theta/\alpha) 2\pi\ell!}
 \int _{-\infty}^\infty e^{-\rmi \tau}  
  \Big(\alpha \int_\lambda^1 e^{ \rmi\tau  y} y^{-\alpha-1} \rmd y\Big)^\ell
 \frac{\rmd \tau}
{ (h(\alpha,\tau,\lambda))^{\theta/\alpha+\ell}}.
 \eea
 
 (ii)\
 Using \eqref{f13} and \eqref{hdef} replace \eqref{f29} with 
           \bea\label{15}
&& 1=
 \frac{\Gamma(\theta)}{ \Gamma(\theta/\alpha)2\pi}
 \int_0^\infty 
x^{\theta/\alpha-1}e^{-x/\Gamma(1-\alpha)} 
 \int _{-\infty}^\infty  e^{-\rmi \tau}
\exp\Big(  \frac{x\alpha}{\Gamma(1-\alpha)}
\int_0^1(e^{\rmi\tau y}-1)y^{-\alpha-1}\rmd y
\Big)
\rmd \tau
\rmd x\cr
&=&
 \frac{\Gamma(\theta)}{ \Gamma(\theta/\alpha)2\pi}
 \int_0^\infty 
x^{\theta/\alpha-1}
 \int _{-\infty}^\infty  e^{-\rmi \tau}
 e^{-xh(\alpha,\tau,1)/\Gamma(1-\alpha)} 
\rmd \tau
\rmd x.
\eea
Change variable from $x/\Gamma(1-\alpha)$ to $x$ and interchange the order of integrations to get the identity  
\be\label{15b}
1=
 \frac{\Gamma(\theta) (\Gamma(1-\alpha))^{\theta/\alpha}}{\Gamma(\theta/\alpha) 2\pi}
 \int _{-\infty}^\infty  e^{-\rmi \tau}
  \int_0^\infty 
x^{\theta/\alpha-1}
 e^{-xh(\alpha,\tau,1)} \rmd x
\rmd \tau
=
 \frac{\Gamma(\theta) (\Gamma(1-\alpha))^{\theta/\alpha}}{ 2\pi}
  \int _{-\infty}^\infty  
 \frac{e^{-\rmi \tau} \rmd \tau}
{(h(\alpha,\tau,1))^{\theta/\alpha}}, 
\ee
where $h(\alpha,\tau,1)
=1-\alpha\int_0^1(e^{\rmi\tau y}-1)y^{-\alpha-1}\rmd y$.

There's an interesting connection between \eqref{15b} and a  formula due to Kevei and Mason \cite{KeveiMason2014}.
They show that, when $\theta/\alpha$ is an integer $\ge 1$, 
\be\label{Darl}
\frac{e^{\rmi \tau}}
{
\big(1-\alpha\int_0^1(e^{\rmi\tau y}-1)y^{-\alpha-1}\rmd y\big)^{\theta/\alpha}}
\ee
is the  characteristic function 
of the limiting distribution  at 0 or $\infty$  of the ratio of a trimmed subordinator 
whose L\'evy tail function is regularly varying with index $-\alpha\in (-1,0)$,  to its largest jump.
See also \cite{IKM2018} and \cite{PY1997}, and 
Darling \cite{Darling1952} 
for a random walk version.
 Thus the integral in \eqref{15b} (divided by $2\pi$) is the value at 2 of the density of the limiting random variable.
    
\section{Appendix:\ Proofs of Lemmas}\label{app}\
  This Appendix 
  contains the proofs of Lemmas  \ref{ylem}, \ref{vcon}, 
  \ref{blem} and a useful technical Lemma \ref{LMu}.
We also give a quick derivation of Theorem 3 of  \cite{MallerShemehsavar2025} and
do some double-checking for  the correctness of some of the  formulae.

 \smallskip\noindent{\bf Proof of Lemma \ref{ylem}:}\ 
 To prove \eqref{6.5a} we  verify the Kallenberg conditions \eqref{210}--\eqref{10b0}  for  the normed sum in \eqref{6.5a}.
Recall \eqref{HM8} and note that $\PP(0<\wh X_{1n}^{(\nu f)}\le n)=1$.
(Again we just write $f$ for $\bff_n$.)
Choose $h\in(0,1)$ and $n$ so large that $\lf hn\rf\ge 1$.
For \eqref{210} consider 
\be\label{21}
  xn^\alpha
\PP\big(\wh X_{1n}^{(\nu f)}\ge \lf hn\rf \big) 
=
\frac{xn^\alpha \sum_{j=\lf hn\rf}^n q_{j}e^{ \nu  f(j/n)} }
{\sum_{j=1}^n q_{j} e^{ \nu  f(j/n)}}
\sim
\frac{x\alpha}{\Gamma(1-\alpha)}
\frac{n^\alpha \sum_{j=\lf hn\rf}^n j^{-\alpha-1} 
e^{ \nu  f(j/n)} }
{\sum_{j=1}^n q_j e^{ \nu  f(j/n)}}.  \quad
\ee
Using the boundedness of $f$ on $[0,1]$ and dominated convergence we see that the sum in the  denominator in  \eqref{21} converges to 1.
The numerator in \eqref{21} has the finite limit
\ben
\lim_{n\to\infty}
n^\alpha \sum_{j=\lf hn\rf}^n j^{-\alpha-1} 
e^{ \nu  f(j/n)} 
=
\int_h^1 e^{\nu  f(y)} y^{-\alpha-1} \rmd y.
\een
To see this, write the lefthand side as 
\ben
\sum_{j=\lf hn\rf}^n \int_{(j-1)/n} ^ {j/n}
(j/n)^{-\alpha-1} e^{ \nu  f(j/n)} \rmd y
=
 \int_{(\lf hn\rf-1)/n} ^ {1}
(\lc ny\rc/n)^{-\alpha-1} e^{ \nu  f(\lc ny\rc/n)} \rmd y
\een
By dominated convergence the righthand integral converges as $n\to\infty$ to 
   $\int_h^1 e^{\nu  f(y)} y^{-\alpha-1} \rmd y$
  for any bounded function $f$ for which 
  $\lim_{n\to\infty} f(\lc ny\rc/n) =f(y)$ almost everywhere; for example, for any $f$ which is continuous almost everywhere on $[0,1]$, thus, for any $f$ of bounded variation on $[0,1]$. 
So the limit of \eqref{21} for such functions  is
\be\label{LMX}
\frac{x\alpha  }{\Gamma(1-\alpha)}
\int_h^1 e^{\nu  f(y)} y^{-\alpha-1} \rmd y
=x
\int_h^1 \mu^{(\nu f)}(\rmd y),
\ {\rm with}\ 
 \mu^{(\nu f)}(\rmd y)= \frac{\alpha}{\Gamma(1-\alpha)}
  e^{\nu  f(y)} y^{-\alpha-1}{\bf 1}_{\{0<y\le 1\}} \rmd y.
\ee
For \eqref{10a0} we look at
  \bea\label{10a}
  &&
  x n^{\alpha-1}
\EE\Big( \wh X_{1n}^{(\nu f)} {\bf 1}_{\{\wh X_{1n}^{(\nu f)}<hn\}}\Big)
\sim
\frac{x\alpha}{\Gamma(1-\alpha)}
\frac{n^{\alpha-1}
 \sum_{j=1}^{\lf hn\rf} j^{-\alpha} e^{ \nu  f(j/n)} }
{\sum_{j=1}^n q_je^{ \nu  f(j/n)}}.
\eea
The sum in the denominator converges to $1$.
For the numerator, again approximate the sum by an integral and get  
the limit of \eqref{10a} as the finite expression
\be\label{10d}
\frac{x\alpha }{\Gamma(1-\alpha)}
\int_0^h e^{\nu  f(y)} y^{-\alpha} \rmd y
=
x\int_0^h y\mu^{(\nu f)}(\rmd y).
\ee
For \eqref{10b0} we get in a similar way
  \be\label{10b}
\lim_{n\to\infty}
x n^{\alpha-2}
\EE\Big(( \wh X_{1n}^{(\nu f)})^2 {\bf 1}_{\{\wh X_{1n}^{(\nu f)}<hn\}}\Big)
=
x\int_0^h y^2\mu^{(\nu f)}(\rmd y).
\ee

With these three conditions satisfied, 
 Corollary 15.16 of \cite{Kallenberg2002} gives \eqref{6.5a},
 where  $Y_x^{(\nu,f)}$ is  infinitely divisible with 
 triplet  $id(a,b,x\mu^{(\nu f)})$,  in Kallenberg's notation. Thus, noting that
 $a=0$ and $b=\int_{0<y<1} y\mu^{(\nu f)}(\rmd y)$, this distribution has characteristic exponent
 (Cor. 15.8, p.291, of \cite{Kallenberg2002})
 \begin{align}\label{10c}
 &
 x\Big(\rmi\tau  b -\tfrac{1}{2} \tau ^2 a 
 +
 \int_{\R\setminus\{0\}}
 (e^{ \rmi\tau  y} -1 -  \rmi\tau  y{\bf 1}_{\{0<y<1}) \mu^{(\nu f)}(\rmd y)\Big) \cr
&
=x
\int_{y>0}
 (e^{ \rmi\tau  y} -1 )\mu^{(\nu f)}(\rmd y)
 =
\frac{\alpha x}{\Gamma(1-\alpha)}
  \int_{0}^1
 (e^{ \rmi\tau  y} -1 ) e^{\nu  f(y)} y^{-\alpha-1} \rmd y, 
 \end{align}
 hence has the  characteristic function 
 in \eqref{ycf}.
  
Next we prove  \eqref{3.num1} using \eqref{3.num0}
 with 
$\phi_{\wh X_n}(\tau/n) 
= \EE\big(\exp(\rmi\tau   \wh X_{n}/n\big)$, $\tau\in\R$,
where we abbreviate 
 $ \wh X_{n}^{(\nu\bfu_n)}$
 to $ \wh X_{n}$. 
 We  use a method of \cite{GK1968} for the local limit theorem.
 As in p.233 of that book, 
 take $A>1$  and $\veps\in(0,1)$, choose $n>A/\veps$    and split the integral on the RHS of \eqref{3.num0}  into
  \be\label{a.3} 
\frac{1}{2\pi } \Big( \int_{|\tau|\le A} 
+ \int_{A<|\tau| \le \veps n} + 
\int_{ \veps n < |\tau| \le \pi n}\Big)
e^{-\rmi \tau}
\big( \phi_{\wh X_n}(\tau/n)\big) ^{ \lf xn^\alpha\rf} \rmd \tau.
 \ee
 By \eqref{6.5a},
$\big( \phi_{\wh X_n}(\tau/n)\big) ^{ \lf xn^\alpha\rf} \to \EE (e^{\rmi \tau Y_x^{(\nu,f)}}) $,
so
the first integral in \eqref{a.3} converges to the corresponding component of the integral in  \eqref{3.num1}.   
    The third integral in \eqref{a.3} is dealt with just as in  \cite{GK1968}, using the fact that $\wh X_n$ is a lattice 
   variable to get
   $ | \phi_{\wh X_n}(\tau/n)| \le e^{-C}$
   for a constant $C>0$, 
   hence
$   | \phi_{\wh X_n}(\tau/n)|^{ \lf xn^\alpha\rf} \le e^{-Cxn^\alpha}$ 
   for $\veps n<|\tau|\le \pi n$.
   See Cor. 2 to Thm. 5 in Section 14 of  \cite{GK1968}.
   Then
   \ben
   \Big| \int_{\veps n<|\tau| \le\pi n}
   e^{-\rmi \tau}
    \big( \phi_{\wh X_n}(\tau/n)\big) ^{ \lf xn^\alpha\rf} \rmd \tau\Big| 
   \le n e^{-Cx n^\alpha}
   \to 0, \ {\rm as}\ n\to\infty.
   \een

 For the second integral in  \eqref{a.3} we prove
  \be\label{limsup0x}
\lim_{A\to\infty} \limsup_{n\to\infty} 
\Big| \int_{A<|\tau| \le \veps n}e^{-\rmi\tau} \big( \phi_{\wh X_n}(\tau/n)\big) ^{ \lf xn^\alpha\rf} \rmd \tau\Big|  =0,
 \ee
 proceeding as follows. Keep $A<|\tau|\le \veps n$. 
 By \eqref{HM8N} we can write
 \ben
 \phi_{\wh X_n}(\tau/n) 
 =
 \frac{1}{D_n}
 \sum_{j=1}^n e^{\rmi \tau j/n}q_j e^{\nu f(j/n)}, 
 \ {\rm with} \
 D_n = \sum_{j=1}^n q_j e^{\nu f(j/n)},
 \een
 where the $q_j$ are given by \eqref{qdef}.
 Since $\sum_{j\ge 1} q_j=1$, $|f(j/n)|\le f_c<\infty$ and $\lim_{n\to\infty}f(j/n)=0$ for each $j>0$, we get 
 $\lim_{n\to\infty}D_n=1$ by  dominated convergence.
 We plan to use the inequality
  \be\label{inal}
 |\phi_{\wh X_n}(\tau/n)| ^{ \lf xn^\alpha\rf} 
=
 \exp\Big(
 \frac{\lf xn^\alpha}{2} \rf
 \log\big(
 1-(1-|\phi_{\wh X_n}(\tau/n) |^2)
 \big)
 \Big)
 \le
 \exp\Big( -\frac{\lf x n^\alpha\rf}{2}  
 \big(1- |\phi_{\wh X_n}(\tau/n)| ^2\big)
 \Big), 
 \ee
 which follows from $-\log(1-z)\ge z$ for $0<z<1$.
 In \eqref{inal},
 \bea\label{inb}
 1-|\phi_{\wh X_n}(\tau/n) |^2
&=&
 1-\EE^2(\cos(\tau\wh X_n/n))- \EE^2(\sin(\tau \wh X_n/n))\cr
 &=&
 \big(  1-\EE(\cos(\tau\wh X_n/n))\big)\big(  1+\EE(\cos(\tau\wh X_n/n))\big)- \EE^2(\sin(\tau \wh X_n/n)).
  \eea

To find a lower bound for 
 $ 1-\EE(\cos(\tau\wh X_n/n))$
 write it as 
  \bean
\EE(1-\cos(\tau\wh X_n/n))=  
 \frac{1}{D_n}
 \sum_{j=1}^n (1-\cos(\tau j/n) ) q_j e^{\nu  f(j/n)},
  \eean
  choose $\veps\in(0,1)$, and keep $1<|\tau|\le \veps n$. Since 
  $q_j \ge c_\alpha j^{-\alpha-1}$ for some
   $c_\alpha>0$   and all $j\ge 1$,  and $|f(j/n)|\le f_c$, 
   so $D_n\le  e^{|\nu|f_c}$,  we get 
    \ben
 \frac{1}{D_n}
 \sum_{j=1}^n (1-\cos(\tau j/n) ) q_j e^{\nu  f(j/n)}
\ge 
 c _\alpha  e^{-|\nu|f_c}\sum_{j=1}^{n/|\tau|} (1-\cos(\tau j/n) )
 j^{-\alpha-1}
  \ge 
   \frac{c_\alpha e^{-|\nu|f_c}\tau^2 }{4n^2} 
  \sum_{j=1}^{n/\tau|}  j^{1-\alpha}
=c _{\alpha,\nu}   \frac{|\tau|^\alpha}{n^{\alpha}}.
  \een
  
To find a lower bound for $1+\EE(\cos(\tau\wh X_n/n))$, write it as $2- \EE(1-\cos(\tau\wh X_n/n))$,  choose $B$ with $\veps B<1$, and estimate 
\bea\label{c1}
1-\EE(\cos(\tau\wh X_n/n))
&=&
 \frac{1}{D_n}
\Big(\sum_{j=1}^B + \sum_{j>B} \Big)(1-\cos(\tau j/n))q_j
e^{\nu  f(j/n)}.
\eea
In the first sum, $\tau j/n\le \tau B/n <\veps B<1$, so  that sum is $\le (1-\cos\veps B)\sum_{j=1}^B  q_je^{|\nu|f_c}\le (\veps B)^2e^{|\nu|f_c}$.
The second sum is $\le \sum_{j>B}q_je^{|\nu|f_c}\le B^{-\alpha}e^{|\nu|f_c}/\alpha$.
Choosing $B=\veps^{-1/4}$ gives an overall upper bound for \eqref{c1} of $(\veps^{3/2}+\veps ^{\alpha/4})e^{|\nu|f_c}=: \veps'_{\nu}$, hence 
a lower bound for $1+\EE(\cos(\tau\wh X_n/n))$ of $2-\veps'_{\nu}$.

 For the sine term in \eqref{inb}, using $|\sin z|\le z$ ($z>0$) and 
   $q_j \le C_\alpha j^{-\alpha-1}$ for some    $C_\alpha>0$, $j\ge 1$,    we get 
  \bean
|\EE(\sin(\tau \wh X_n/n))|
  &\le&
 e^{|\nu|} \sum_{j=1}^n |\sin(\tau j/n)|q_j 
  \le 
  C_\alpha  e^{|\nu|f_c} 
  \Big(\frac{|\tau|}{n} \sum_{j=1}^{n/|\tau|}  j^{-\alpha}
  +
\sum_{j\ge n/|\tau|}j^{-1-\alpha} \Big) 
=C _{\alpha,\nu}   \frac{|\tau|^\alpha}{n^{\alpha}}.
  \eean
Since $|\tau|\le \veps n$, for another constant $\wtilde C _{\alpha,\nu}$ 
we have  $ \EE^2(\sin(\tau \wh X_n/n))
\le\wtilde C _{\alpha,\nu} \veps^\alpha |\tau|^\alpha n^{-\alpha}$, 
 so, putting all parts together,
 we have a lower bound for the RHS of \eqref{inb} of the form  $c(\alpha,\nu)  |\tau|^\alpha n^{-\alpha}$, where 
 $c(\alpha,\nu)>0$,
and consequently an upper  bound for the RHS of \eqref{inal} of
$   |\phi_{\wh X_n}(\tau/n)| ^{ \lf xn^\alpha\rf} 
   \le 
e^{-xc(\alpha,\nu) |\tau|^\alpha}$. 
 \eqref{limsup0x} follows and this 
 shows that the second  integral in \eqref{a.3}  can be neglected for this model.

 We show finally that the integral  
 $\int_{-\infty}^{\infty} e^{-\rmi \tau}
 \EE (e^{\rmi\tau Y_x^{(\nu,f)}}) \rmd \tau$  in \eqref{3.n} is absolutely convergent for all $x>0$.
In fact this is the case with $e^{-\rmi \tau}$ replaced by $e^{-\rmi \tau z}$, for any $z\in\R$, so  each $Y_x^{(\nu,f)}$ has a continuous density, $f_{Y_x^{(\nu,f)}}(z)$, for $z\in\R$.
  To prove it,  choose $|\tau|>1$
   and write,   with  $c(\alpha,\nu)>0$,
 \bea\label{a1}
 |\EE(e^{\rmi \tau Y_x^{(\nu,f)}})|
 &=&
 \exp\Big(-\frac{\alpha x}{\Gamma(1-\alpha)}
  \int_{0}^1
 (1-\cos(\tau  y) ) e^{\nu  f(y)} y^{-\alpha-1} \rmd y\Big)\cr
 &=&
  \exp\Big(-\frac{\alpha x|\tau|^{\alpha}}{\Gamma(1-\alpha)}
  \int_{0}^{|\tau|}(1-\cos y)
 e^{\nu f(y/|\tau|)} y^{-\alpha-1} \rmd y\Big)\cr
 &\le&
  \exp\Big(-\frac{\alpha x |\tau|^{\alpha}e^{-|\nu|f_c }}
  {\Gamma(1-\alpha)} \int_{0}^{1}  (1-\cos y) y^{-\alpha-1} \rmd y\Big) 
=
e^{-c(\alpha,\nu)x|\tau|^\alpha}.
  \eea
Thus the LHS is bounded by an integrable function of $\tau$ and we have the absolute convergence. This completes the proof of Lemma \ref{ylem}.
\halmos

 \medskip\noindent{\bf Proof of Lemma \ref{vcon}.}\ 
 Recalling \eqref{dvN2} and   \eqref{qdef}, we have 
  \be\label{dv2N}
   \PP\Big(V_{1n}=u_{jn}=f(j/n)\Big) 
=
   q_{jn}
   =
   \frac{\alpha \Gamma(j-\alpha)}{q_{+n}\Gamma(1-\alpha)j!},\ 1\le j\le n,
     \ee
where we now omit the superscript $\bfu_n$ on $V_{1n}^{(\bfu_n)}$. 
 In \eqref{HM8} and following we assumed
  $b_n= k_n$.
 The mgf of the expression in \eqref{nsgN}, in which we can omit
 $b_n/k_n$ and $E(V_{1n})$, 
  is then
 \bea\label{mgd0}
 &&
 \big( \EE(e ^{\nu  V_{1n}})\big)^{\lf x n^\alpha\rf}
=
   \exp\Big(\lf x n^\alpha\rf
  \Big(\log\big(1+\sum_{j=1}^n q_{jn} (e ^{\nu  f(j/n) }
  -1)\big) \Big)\Big)\cr
  &&=
     \exp\Big(\lf x n^\alpha\rf
  \Big(\sum_{j=1}^n q_{jn} (e ^{\nu  f(j/n) }  -1)
  +O\big(R_n^2\big)\Big)\Big),
  \eea
  where 
  \ben
  R_n = \sum_{j=1}^n q_{jn} (e ^{\nu | f(j/n)| }  -1)
  =O\Big( \sum_{j=1}^n q_{jn}|f(j/n)| \Big)
  =O\Big(n^{-\alpha} \int_0^1 y^{-\alpha-1} |f(y)| \rmd y\Big).
  \een   
Letting $n\to\infty$ and approximating the sum by an integral, 
the RHS of \eqref{mgd0}  has limit
equal to  the expression in \eqref{vfin}. 
This completes the proof of Lemma \ref{vcon}.
\halmos
     
     \smallskip\noindent{\bf Proof of Lemma \ref{blem}.}\      
To see that the limit in \eqref{f29} equals 1, note that, for $K>0$,
\bea\label{13}
\PP(K_n(\alpha,\theta)\le Kn^\alpha)
&=&
 \sum_{k=1}^{\lf  Kn^\alpha \rf}\PP(K_n(\alpha,\theta)=k)
 =
 n^\alpha  \sum_{k=1}^{\lf  Kn^\alpha \rf}
\int_{k/n^\alpha} ^{(k+1)/n^\alpha} \PP(K_n(\alpha,\theta)=\lf xn^\alpha\rf  )\rmd x \cr
&=&
n^\alpha \int_{1/n^\alpha} ^{K+1/n^\alpha} \PP(K_n(\alpha,\theta)=\lf xn^\alpha\rf  ) \rmd x.
\eea
Since $K_n(\alpha,\theta)/n^\alpha$ converges to a proper
(Mittag-Leffler) rv, the 
limit of the LHS of \eqref{13} as $n\to\infty$ then $K\to\infty$ is 1, and so the same is true of the limit of  the RHS, and this is the same limit found in \eqref{f29}. 
For the remainder of \eqref{f29},    
we  deduce  from \eqref{3b}, and Fatou's lemma,
           \bea\label{15n}
&&  \lim_{n\to\infty}
n^\alpha\int_0^\infty  \PP\big(K_n(\alpha,\theta)= \lf xn^\alpha\rf  \big)\rmd x
\ge 
 \frac{\alpha\Gamma(\theta)}{\Gamma(\theta/\alpha)}
\int_0^\infty   x^{\theta/\alpha}g_\alpha(x)
\rmd x.
     \eea
     Here the LHS equals 1 as just shown -- but also  the RHS equals 1 because the $ x^{\theta/\alpha}$--moment of the Mittag-Leffler  rv is precisely equal to 
     $\Gamma(\theta/\alpha)/\alpha\Gamma(\theta)$
     (e.g., \cite{Pitman2006}, p.68).
     Thus equality holds in \eqref{15n}.
     Substituting on the right of this equality for the $g_\alpha(x)$ term in \eqref{3b} and 
recalling $ \EE(e^{\rmi \tau Y_x^{(0)}})$ in \eqref{f13} 
     gives  \eqref{f29}.
     This completes the proof of Lemma \ref{blem}.
      \halmos.  

\medskip\noindent{\bf Completion of the Proof of Theorem \ref{4para}:}\ 
It remains to show that the RHS of \eqref{f28} defines a proper probability mass function.
The identity
\be\label{zid}
\sum_{\ell=0}^\infty \frac{\Gamma(\theta/\alpha +\ell)}{\ell!} z^\ell
= \frac{\Gamma(\theta/\alpha)}{(1-z)^{\theta/\alpha}}, \ {\rm for}\ |z|<1, \theta>0, \alpha>0,
\ee
is easily proved. From \eqref{51} write 
\be\label{21a}
      \lim_{n\to\infty}     \PP\big(S_n^\lambda(\alpha,\theta) =\ell\big)
=
\frac{\Gamma(\theta)\Gamma(\theta/\alpha+\ell) (\Gamma(1-\alpha))^{\theta/\alpha}}
{ \Gamma(\theta/\alpha)2\pi \ell! }
\int _{-\infty}^\infty
 \frac{ e^{-\rmi \tau}
z^\ell}{(h(\alpha,\tau,\lambda))^{\theta/\alpha}} \rmd \tau,
 \ee
 where (recalling  \eqref{hdef})
 \be\label{zdef}
z=
\frac{\alpha \int_\lambda^1 e^{ \rmi\tau  y} y^{-\alpha-1} \rmd y}
{h(\alpha,\tau,\lambda)}
=
\frac{\alpha \int_\lambda^1 (e^{ \rmi\tau  y}-1) y^{-\alpha-1} \rmd y+\lambda^{-\alpha}-1}
{ \lambda^{-\alpha}-\alpha 
\int_0^\lambda \big(e^{ \rmi\tau  y}-1\big) y^{-\alpha-1} \rmd y}.
 \ee
 Then we can calculate
  \be\label{zD}
1-z
=\frac{ 1-\alpha 
\int_0^1 \big(e^{ \rmi\tau  y}-1\big) y^{-\alpha-1} \rmd y}
{ \lambda^{-\alpha}-\alpha 
\int_0^\lambda \big(e^{ \rmi\tau  y}-1\big) y^{-\alpha-1} \rmd y}
=\frac{ 1-\alpha 
\int_0^1 \big(e^{ \rmi\tau  y}-1\big) y^{-\alpha-1} \rmd y}
{h(\alpha,\tau,\lambda)}
=\frac{ h(\alpha,\tau,1)}
{h(\alpha,\tau,\lambda)}.
\ee
 Add over $\ell$ in \eqref{21a} using \eqref{zid} to get, for the RHS,
 \be\label{22}
 \frac{\Gamma(\theta)(\Gamma(1-\alpha))^{\theta/\alpha}}
{ 2\pi}
\int _{-\infty}^\infty  e^{-\rmi \tau} 
\frac{\rmd \tau}
{ \big(1-\alpha 
\int_0^1 \big(e^{ \rmi\tau  y}-1\big) y^{-\alpha-1} \rmd y\big)^{\theta/\alpha}}.
 \ee
By \eqref{15b}   this equals 1.
\halmos

\medskip
\noindent{\bf A Basic Lemma.}\
The  following lemma summarises  a method used in    \cite{IpsenMallerShemehsavar2021}, p.375,   
\cite{MallerShemehsavar2025}, and herein.
 
\begin{lemma}\label{LMu}
  Let $J,k,n, (m_i)_{1\le i\le n}$ be integers
  with $0\le J< n$, $1\le k\le n$, and 
  $0\le m_i\le n$, $1\le i\le n$.
Let $m_{+}^{(J)}= \sum_{j=1}^Jm_j$ and
$m_{++}^{(J)}= \sum_{j=1}^Jjm_j$, with
$m_{+}^{(0)}=m_{++}^{(0)}=0$, 
$m_{+}=m_{+}^{(n)}=k$, $m_{++}=m_{++}^{(n)}=n$,
$k^{(J)}=k-m_{+}^{(J)}$ and 
$n^{(J)}=n-m_{++}^{(J)}$.
Let  $\mu_j$, $J+1\le j\le n$,  be positive constants with
$\mu_+^{(J)}=\sum_{j=J+1}^n \mu_j$.
Let $\bfm^{(J)}= (m_{J+1},\ldots,m_n)$ and 
  \ben
  A_{kn}^{(J)}=\Big\{m_j\ge 0, J+1\le j\le n: \,
 \sum_{j=J+1}^nm_j= k^{(J)},\, 
\sum_{j=J+1}^n jm_j=n- m_{++}^{(J)}\Big\},
\een
and let 
 $\big(X_{in}^{(J)}\big)_{1\le i\le  k^{(J)}}$ be i.i.d. with
\be\label{LM2}
\PP\big(X_{1n}^{(J)}=j\big) =
\frac{\mu_j}{\mu_+^{(J)}},\ J+1\le j\le n.
\ee
Then 
\be\label{LM1}
\sum_{\bfm^{(J)}\in A_{kn}^{(J)}}
\prod_{j=J+1}^n \frac{ \mu_j^{m_j}}{m_j!}
=
\frac{\big(\mu_+^{(J)} \big)^{ k^{(J)}}}
{k^{(J)}!}
 \PP\Big(\sum_{i=1}^{ k^{(J)}} X_{in}^{(J)}= n^{(J)}\Big).
\ee
\end{lemma}
  
 \noindent{\bf  Proof of Lemma \ref{LMu}:}\
   Let $N_j$ be independent Poisson $(\mu_j)$ random variables, $J+1\le j\le n$, so that
  \be\label{LJ1}
  \PP(N_j=m_j, J+1\le j\le n)=
  \prod_{j=J+1}^n \frac{ \mu_j^{m_j}e^{-\mu_j}}{m_j!}.
  \ee
  Adding the RHS of \eqref{LJ1} over the integers in $A_{kn}^{(J)}$ we get
  \bea\label{LJ2}
&&  \sum_{\bfm^{(J)}\in A_{kn}^{(J)}}
\prod_{j=J+1}^n \frac{ \mu_j^{m_j}e^{-\mu_j}}{m_j!}
=
 \PP\Big(\sum_{j=J+1}^n jN_j= n^{(J)},
 \sum_{j=J+1}^n N_j= k^{(J)}\Big)\cr&&\cr
 &=&
  \PP\Big(\sum_{j=J+1}^n jN_j= n^{(J)}\Big|
 \sum_{j=J+1}^n N_j= k^{(J)}\Big)
  \PP\Big( \sum_{j=J+1}^n N_j= k^{(J)} \Big).
  \eea
  Here $ \sum_{j=J+1}^n N_j$ is distributed as Poisson with mean $ \sum_{j=J+1}^n \mu_j=\mu_+^{(J)}$.
  Substituting  in \eqref{LJ2} for
  \ben
 \PP\Big( \sum_{j=J+1}^n N_j= k^{(J)} \Big)=
 \frac   {\big(  \mu_+^{(J)} \big)^{ k^{(J)}}} { k^{(J)}!}
     e^{-\mu_+^{(J)}}
  \een
   gives
    \be\label{LJ3}
  k^{(J)}! \sum_{\bfm^{(J)}\in A_{kn}^{(J)}}
\prod_{j=J+1}^n \frac{\mu_j^{m_j}}{m_j!}
=\big(  \mu_+^{(J)} \big)^{ k^{(J)}}
  \PP\Big(\sum_{j=J+1}^n jN_j= n^{(J)}\Big|
 \sum_{j=J+1}^n N_j= k^{(J)}\Big).
  \ee  
  Conditional on $ \sum_{j=J+1}^n N_j= k^{(J)}$,
  the vector 
  $(N_j, J+1\le j\le n)$ has the  distribution 
of  a multinomial  vector    $\bfMult(J, k^{(J)}, n, \bfp_n^{(J)})$ 
with elements  $(M_j)_{J+1\le j\le n}$ and mass function 
\be\label{M0}
\PP\big(\bfMult(J, k^{(J)}, n, \bfp_n^{(J)})=(m_{J+1},\ldots, m_n) \big)=
 k^{(J)}!
\prod_{j=J+1}^n
\frac{p_{j}^{m_j}}{m_j!},
\ee
  where $p_j=\mu_j/\mu_+^{(J)}$, $J+1\le j\le n$, and $\bfp_n^{(J)}=(p_{J+1}, \ldots, p_n)$.
  Thus, conditional on $ \sum_{j=J+1}^n N_j= k^{(J)}$,
   $\sum_{j=J+1}^n jN_j$ has the distribution of 
 $\sum_{j=J+1}^n jM_j$.
 The mgf of $\bfMult(J, k^{(J)}, n, \bfp_n^{(J)})$ is
 \ben
\EE\Big(\prod_{j=J+1}^n e^{\nu_jM_j}\Big)
 =
 \Big(\sum_{j=J+1}^np_j e^{\nu_j}\Big)^{ k^{(J)}}
 \een
  where $\nu_j\in\R$, $J+1\le j\le n$.
  Setting $\nu_j=\nu j$, $\nu\in\R$, gives
   \ben
 \EE\Big(\exp\Big(\nu \sum_{j=J+1}^n j M_j\Big)\Big)
 =
 \Big(\sum_{j=J+1}^np_j e^{\nu j}\Big)^{ k^{(J)}},
 \een
 and the RHS is 
the mgf of the sum of $ k^{(J)}$ independent rvs $X_{in}^{(J)}$  with the distribution in \eqref{LM2}.
 Consequently 
 $\sum_{j=J+1}^n jM_j \eqdr \sum_{i=1}^{ k^{(J)}} X_{in}^{(J)}$
 and by \eqref{LJ3} and \eqref{M0}
 \bean
   \PP\Big(\sum_{j=J+1}^n jN_j= n^{(J)}\Big|
 \sum_{j=J+1}^n N_j= k^{(J)}\Big)
&=&
\PP\Big(\sum_{j=J+1}^n jM_j= n^{(J)}\Big)=
\PP\Big( \sum_{i=1}^{ k^{(J)}} X_{in}^{(J)}= n^{(J)}\Big),
  \eean
  and so \eqref{LM1} is proved.
\halmos 
     
  \medskip\noindent {\bf Remarks.}\
 Lemma \ref{LMu} is a  key component of the proofs. 
     Comparing our methods with those of \cite{ABT2003}, a main difference is that we condition on sums of Poisson variables  like
     $\sum_j N_j$, whereas \cite{ABT2003} condition on sums like 
      $\sum_j jN_j$. Both approaches lead to local limit theorems, but not the same ones,  for sums of independent rvs.
       In particular expressions like our $f_{Y}(1)$ occur frequently in \cite{ABT2003}, and ratios of random walk probabilities such as in \eqref{ecenN} have been studied, though generally in a more restricted setting; see for example \cite{KestenRatioII1963} and his references where the proof techniques if not the results themselves may be applicable.
 
\bigskip\noindent{\bf Theorem 3 of \cite{MallerShemehsavar2025} follows from Theorem \ref{thm2N}:}\
Fix an integer $J$, $1\le J<n$,  let $\bfu_J=(u_1, u_2, \ldots,u_J)$ be an arbitrary vector in $\R^J$, and let $\bf0$ be an $n-J$ vector of zeroes.
Applying Theorem \ref{thm2N} with $\bfu_n^T=(\bfu_J^T\, \bf0^T)$ 
gives a formula for  the conditional  mgf of 
$\bfM_{Jn}:=(M_{1n}, \ldots, M_{Jn})$
in the form
  \be\label{ecenNX}
\EE\Big(
\exp\Big(\nu\bfu_J^T\Big(\frac{\bfM_{Jn}}{K_n}-\bfq_J\Big)\Big)\Big| K_n=k\Big)=
\EE\Big(\exp\Big(
  \frac{ \nu }{k} \sum_{i=1}^{k}
   \big(V_{in}^{(\bfu_n)}- \EE(V_{in}^{(\bfu_n)})\big)\Big)\Big)\times
\frac{\PP\Big(\sum_{i=1}^{k} \wh X_{in}^{(\nu\bfu_n)}
=n\Big)}
{ \PP\Big(\sum_{i=1}^{k} \wh X_{in}^{(0)}=n\Big)},
 \ee
 for  $\nu\in\R$ and $k\in \N_n$. 
 Set $k=k_n(x) =\lf xn^\alpha\rf$, $x>0$,  and $\nu=\nu_n=n^{\alpha/2}$.
 We  know then that $\lim_{n\to\infty} \PP(K_n=\lf xn^\alpha\rf)$ exists,  finite and positive.
Recall   $q_j = \alpha \Gamma(j-\alpha)/j!\Gamma(1-\alpha)$ and let  $\bfQ_J$ be  the $J\times J$ matrix  with 
diagonal elements $q_{j}(1-q_{j})$ and 
off-diagonal elements $-q_{j} q_{\ell }$, $1\le j\ne \ell \le J$.
  In this special case the 
  $(V_{in}^{(\bfu_n)})_{1\le i\le k}$ in \eqref{ecenNX} are i.i.d. with 
 \be\label{dvNX}
\begin{cases}
 P(V_{1n}^{(\bfu_n)}= u_{j})=q_{jn}, & \text{if } 1\le j\le J;\\
 P(V_{1n}^{(\bfu_n)}=0)=q_{jn},  & \text{if } J+1\le j\le n,
\end{cases}
\ee
 the 
 $\big(\wh X_{in}^{(\nu{\bfu_n}  )}\big)_{1\le i\le k}$   are i.i.d. with the distribution 
\be\label{HM8NnX}
\PP\big(\wh X_{1n}^{(\nu{\bfu_n}  )}=j\big) 
=
\begin{cases}
\displaystyle{\frac{q_{j}e^{ \nu u_{j}/k} }{D_n}} & \text{if } 1\le j\le J;\\
\displaystyle{\frac{q_{j} }
{D_n}}  & \text{if } J+1\le j\le n.
\end{cases}
\ee
where $D_n=\sum_{\ell=1}^J q_{\ell } e^{ \nu u_{\ell }/k}+
\sum_{\ell=J+1}^n q_{\ell }$,
and the 
 $\big(\wh X_{in}^{(0 )}\big)_{1\le i\le k}$ are i.i.d. with
\be\label{H0X}
\PP\big(\wh X_{1n}^{(0 )}=j\big) 
= q_{jn} =
\frac{q_{j} }
{\sum_{\ell=1}^n q_{\ell } },\ 1\le j\le n.
\ee
(Compare  \eqref{HM8Nn}, \eqref{H0}, \eqref{dvN}.) 
     The $V_{in}^{(\bfu_n)}$ have 
      $\EE(V_{1n}^{(\bfu_n)})= \bfu_J^T\bfq_J$
    and 
       ${\rm Var}(V_{1n}^{(\bfu_n)})= \bfu_J^T\bfQ_J\bfu_J$.
With $k=\lf xn^\alpha\rf$ and $\nu=n^{\alpha/2}$,
the exponent in the first  term on the RHS of \eqref{ecenNX} is 
\ben
\frac{n^{\alpha/2}}{\lf xn^\alpha\rf}
 \sum_{i=1}^{\lf xn^\alpha\rf}
   \big(V_{in}^{(\bfu_n)}- \EE(V_{in}^{(\bfu_n)})\big)
\een
where the sum is of  a triangular array of independent rvs.
The whole expression has mean 0 and  variance 
$x^{-1} \bfu_J^T\bfQ_J\bfu_J$ and is  easily seen to be asymptotically normal with these parameters.
Since $\bfu_J$ is arbitrary, using \eqref{dvNX}  this reproduces the result in Theorem 3 of \cite{MallerShemehsavar2025} 
(the conditional asymptotic normality of a finite number of the $(M_{jn}/K_n)_{1\le j\le J}$ after centering and an $n^{\alpha/2}$ norming) 
provided we show the ratio of probabilities  on the RHS of \eqref{ecenNX} has limit 1.
This follows because, for $h\in (0,1)$, it's easy to check that the limits in \eqref{LMX}, \eqref{10d} and \eqref{10c} hold as stated but with $f(y)$ set equal to 0. Thus \eqref{6.5a} holds with $Y_x^{(\nu,f)}$ replaced by 
 $Y_x^{(0)}$ and consequently \eqref{3.n} holds in the present case in the form 
    \be\label{3.nX}
    \lim_{n\to\infty}
n\PP\Big(\sum_{i=1}^{\lf xn^\alpha\rf} \wh X_{in}^{(\nu \bfu_n)}=n\Big)
= f_{Y_x^{(0)}} (1),
 \ee
where the RHS is finite and positive. Exactly the same  result holds with  $ \wh X_{in}^{(\nu \bfu_n)}$  replaced by 
$ \wh X_{in}^{(0)}$, and so  the ratio of probabilities  on the RHS of \eqref{ecenNX} indeed has limit 1.
This completes the demonstration. \halmos

\newpage
\bigskip\noindent{\bf Some Double Checking.}\

\smallskip\noindent{\bf (a) Check  
 \eqref{ecenN} when $\bfu_n= {\bf 1_n}$.}  
 We can equivalently show that \eqref{cN2}, which we write as
 \bea\label{a1.1}
 &&
 \EE\Big( \exp\Big(\nu 
 \bfu_n^T\frac{\bfM_n}{K_n} \Big);K_n=k \Big)
 =
\frac{C_{nk}q_{+n}^{k}}{k!}
 \PP\Big(\sum_{i=1}^{k} \wh X_{in}^{(\nu\bfu_n)}=n\Big)
 \Big(\sum_{j=1}^n q_{j }e^{ \nu u_{nj}/k}  \Big)^{k},
 \eea
 reduces correctly when  $\bfu_n= {\bf 1_n}$.
So assume $\bfu_n= {\bf 1_n}$. Then 
 $\bfu_n^T\bfM_n =\sum_{j=1}^nM_{jn}=K_n$, and the LHS of \eqref{a1.1} is
$e^{\nu } \PP(K_n=k )$,
 which
  equals
\be\label{a.40}
e^{\nu } \frac{C_{nk}q_{+n}^{k}}{k!}  
 \PP\Big(\sum_{i=1}^{k} \wh X_{in}^{(0)}=n\Big)
 \ee
  by \eqref{M4.2N1}. 
 On the other hand  the RHS of \eqref{a1.1} with  $\bfu_n= {\bf 1_n}$ equals
\bea\label{a.50}
\frac{C_{nk}q_{+n}^{k}}{k!}
 \PP\Big(\sum_{i=1}^{k} \wh X_{in}^{(\nu,\bf1_n)}=n\Big)
e^{ \nu }.
 \eea
 But when $\bfu_n =\bf1_n$, by   \eqref{3.num0} and \eqref{HM8Nn}, 
 \bean
  \PP\Big(\sum_{i=1}^{k} \wh X_{in}^{(\nu,\bf1_n)}=n\Big)
&=&
 \frac{n}{2\pi } \int_{-\pi}^{\pi} e^{-\rmi n \tau}
 \big(  
 \EE (e^{\rmi \tau \wh X_{1n}^{(\nu,\bf1_n)}}) \big)^k
 \rmd \tau 
 =
 \frac{n}{2\pi } \int_{-\pi}^{\pi} e^{-\rmi n \tau}
\Big(  
 \frac{ \sum_{j=1}^n q_{j} e^{ \nu/k} e^{\rmi \tau j} }
{ \sum_{j=1}^n q_{j} e^{ \nu/k }}\Big)^k
 \rmd \tau\cr
 &=&
 \frac{n}{2\pi } \int_{-\pi}^{\pi} e^{-\rmi n \tau}
\Big(  
 \frac{ \sum_{j=1}^n q_{j} e^{\rmi \tau j} }
{ \sum_{j=1}^n q_{j} }\Big)^k
 \rmd \tau      
=
   \PP\Big(\sum_{i=1}^{k} \wh X_{in}^{(0)}=n\Big).
 \eean
 Thus the expressions in \eqref{a.50} and \eqref{a.40} are equal. \halmos

\smallskip\noindent{\bf (b) Check \eqref{STA} 
when $p=1$.}   
Set $f(y)=y$  and  replace $\nu$ by $\rmi\nu$ in 
  \eqref{f1} to get 
  \be\label{p=1}
  f_{Y_x^{(\rmi\nu,1)}}(1)
  =
  \frac{1}{2\pi}\int _{-\infty}^\infty e^{-\rmi \tau}
  \exp\Big(\frac{\alpha x}{\Gamma(1-\alpha)}
  \int_{0}^1
 \big((e^{ \rmi(\tau +\nu ) y} -1 ) 
 -(e^{\rmi \nu y}-1)\big)
 \frac{\rmd y}{  y^{\alpha+1} }
 \Big) \rmd \tau,
 \ee
  regarding 
 $  f_{Y_x^{(\rmi\nu,1)}}(1)$ just as an abbreviation for the 
 RHS of \eqref{p=1}. 
Change variable from $\tau+\nu $ to $\tau$ and write that RHS expression as
 \be
  = \frac{e^{\rmi  \nu }}{2\pi}\int _{-\infty}^\infty e^{-\rmi \tau}
 \frac{  \exp\Big(\frac{\alpha x}{\Gamma(1-\alpha)}
  \int_{0}^1
(e^{ \rmi\tau y} -1 )  y^{-\alpha-1} \rmd y\Big)
 }
 {  
 \exp\Big(\frac{\alpha x}{\Gamma(1-\alpha)}
  \int_{0}^1
  (e^{\rmi \nu y}-1)\big)
 y^{-\alpha-1} \rmd y
 \Big)
 }
   \rmd \tau.
 \ee
We can rewrite this as
  \bea\label{p=1a}
  f_{Y_x^{(\rmi\nu,1)}}(1)
 = \frac{e^{\rmi  \nu }}{2\pi}\int _{-\infty}^\infty e^{-\rmi \tau}
 \frac{  
\EE(e^{ \rmi\tau T_\alpha})
 }
 {  \EE  (e^{\rmi \nu T_\alpha}) }   \rmd \tau
 = \frac{
e^{\rmi  \nu } f_{T_\alpha}(1) 
  }
 {  \EE  (e^{\rmi \nu T_\alpha}) },
 \eea
 where $T_\alpha$ has a stable L\'evy measure but truncated to $(0,1]$, thus
 \ben
  \EE  (e^{\rmi \tau T_\alpha}) = \exp\Big(\frac{\alpha x}{\Gamma(1-\alpha)}
  \int_{0}^1
(e^{ \rmi \tau y} -1 )  y^{-\alpha-1} \rmd y\Big).
\een
Similarly we get
  \be\label{p=1b}
  f_{Y_x^{(0)}}(1)=  \frac{1}{2\pi}\int _{-\infty}^\infty 
  e^{-\rmi \tau}
  \exp\Big(\frac{\alpha x}{\Gamma(1-\alpha)}
  \int_{0}^1
  (e^{ \rmi \tau y} -1 )  y^{-\alpha-1} \rmd y
  \Big)
  \rmd \tau=f_{T_\alpha}(1).
  \ee
    
  The mgf on the RHS of   \eqref{STA} 
  when $p=1$ is, by    \eqref{mgd5} 
  (replacing $\nu$ by $\rmi\nu$),
  \bean
    &&\EE  (e^{\rmi \nu T_{\alpha/1}})
    = e^{-\rmi \nu    x/\Gamma(1-\alpha)}
      \exp\Big(\frac{\alpha x}{\Gamma(1-\alpha)}
  \int_{0}^1
  (e^{ \rmi \nu y} -1 )  y^{-\alpha-1} \rmd y  \Big)
  =
 e^{-\rmi \nu  x /\Gamma(1-\alpha)}
 \EE  (e^{\rmi \nu T_\alpha}).
  \eean
  Thus the RHS of  \eqref{STA}   with $p=1$ is
  \bea\label{p=1c}
  \frac{  f_{Y_x^{(\rmi \nu,1)}}(1)}{f_{Y_x^{(0)}}(1)}
  \EE  (e^{\rmi \nu S_{\alpha/1}}) 
  &=&
   \frac{
 e^{\rmi  \nu } f_{T_\alpha}(1) 
  }
 {  \EE  (e^{\rmi \nu T_\alpha}) f_{T_\alpha}(1)  }
 \times
  \EE  (e^{\rmi \nu   T_\alpha})=
   e^{\rmi  \nu }.
  \eea
 Since $\sum_{j=1}^n jM_{jn}=n$, when $p=1$ the LHS of  
 \eqref{STA}  
 (replacing $\nu$ by $\rmi\nu$)
 is
 \bean
\lim_{n\to\infty}  \EE  \big( \exp(\rmi\nu) \big|K_n=\lfloor x n^\alpha\rfloor \big)
=
e^{\rmi \nu},
\eean
 equal to the RHS of \eqref{p=1c}.
\halmos
 
  {}

\bigskip
\noindent 
R.A. Maller, 
Research School of Finance, Actuarial Studies \& Statistics\\
 The Australian National University, Canberra, ACT, 0200, Australia.\\
Email: Ross.Maller@anu.edu.au\\

\vskip-0.2cm
\noindent 
S. Shemehsavar, College of Science, Technology, Engineering
 \& Mathematics\\
  Murdoch University, Perth, Western Australia\\
Email: soudabeh.shemehsavar@murdoch.edu.au  (corresponding author).\\
  \end{document}